\documentclass[12pt,reqno]{amsart}

\usepackage[all,cmtip]{xy}
\usepackage{amsmath,amssymb}
\usepackage{amscd,latexsym}

\newtheorem{theorem}{Theorem}[section]
\newtheorem{lemma}[theorem]{Lemma}
\newtheorem{proposition}[theorem]{Proposition}
\newtheorem{corr}[theorem]{Corollary}

\newtheorem{conjecture}[theorem]{Conjecture}

\theoremstyle{definition}

\newtheorem{definitions and remarks}[theorem]{Definitions and Remarks}

\theoremstyle{remark}

\title{On representations of ${\rm GL}_{2n}(F)$ with a symplectic period}
\author{Arnab Mitra}

\numberwithin{equation}{section}

\begin{document}

\begin{abstract}
The main aim of this paper is to classify the irreducible admissible representations of ${\rm GL}_{4}(F)$ and ${\rm GL}_{6}(F)$ for a non-archimedean local field $F$, which bear a non-trivial linear form invariant under the groups ${\rm Sp_{2}}(F)$ and ${\rm Sp_{3}}(F)$ respectively. We propose a few conjectures for the general case.  
\end{abstract}
\maketitle
\section {Introduction}\setcounter{equation}{0}\label{Introduction}

Let $G={\rm GL}_{2n}(F)$ for $F$ a non-archimedean local field of characteristic 0 and $H$ a symplectic subgroup of $G$ of rank $n$.
A representation $\pi$ of $G$ is said to have a symplectic period (or to be $H$-distinguished) if ${\rm Hom}_{H}(\pi|_{H} ,\mathbb C)\neq 0$. This paper gives a complete list of irreducible admissible representations of ${\rm GL}_{4}(F)$ and ${\rm GL}_{6}(F)$ having a symplectic period. We also make a few conjectural statements for ${\rm GL}_{2n}(F)$ at the end.

The motivation for this problem comes from the work of Klyachko in the case of finite fields \cite{Klyachko}. He found a set of representations generalizing the Gelfand-Graev model after which Heumos and Rallis (in \cite{Heumos-Rallis}) studied the analogous notion in the $p$-adic case. Moreover, they proved multiplicity one theorems in the symplectic case. 

Continuing the works of Heumos and Rallis, Offen and Sayag proved in a series of papers (\cite{Offen-Sayag1}, \cite{Offen-Sayag2}, \cite{Offen-Sayag3}), the uniqueness property of the Klyachko models and multiplicity one results for irreducible admissible representations. They also showed the existence of the Klyachko model for unitary representations. To state the results precisely we need to introduce notation. 

Let $\delta$ be a square integrable representation of ${\rm GL}_{r}(F)$. Denote by $U(\delta,m)$ the unique irreducible quotient of the representation, $$\nu^{(m-1)/2}\delta \times \nu^{(m-3)/2}\delta \times\cdots\times\nu^{-(m-1)/2}\delta.$$ Then we have the following proposition due to Offen and Sayag.

\begin{proposition}[{\bf Offen, Sayag}, \cite{Offen-Sayag1}]\label{main prop}
Let $\delta_{i}$'s be square integrable representations of ${\rm GL}_{r_{i}}(F)$ and $m_{i}$'s be positive integers for $i=1,...,t$. Let $\chi_{i}$ be a character of ${\rm GL}_{2m_{i}r_{i}}(F)$. Then the representation, $$ \chi_{1}U(\delta_{1},2m_{1})\times\cdots\times\chi_{t}U(\delta_{t},2m_{t})$$ has a symplectic period.
\end{proposition}
Further define, $$\mathcal B=\{U(\delta,2m), \nu^{\alpha}U(\delta,2m)\times\nu^{-\alpha}U(\delta,2m)\}$$where $\delta$ varies over the discrete series representations and $\alpha\in\mathbb R$ such that $|\alpha|<1/2$. Then we have the following theorem which classifies unitary representations having a symplectic period.
\begin {theorem}[{\bf Offen, Sayag}, \cite{Offen-Sayag2}]\label{unitary symplectic}
Let $\pi=\tau_{1}\times\cdots\times\tau_{r}$ such that $\tau_{i}\in \mathcal B$. Then $\pi$ has a symplectic period. Conversely, if $\pi$ is an irreducible unitary representation with a symplectic period then there exist $\tau_{1},...,\tau_{r}\in \mathcal B$ such that $\pi=\tau_{1}\times\cdots\times\tau_{r}$.
\end{theorem}

A natural question now is to classify all irreducible admissible representations which admit a symplectic model. For ${\rm GL}_{4}(F)$ we have the following result.
\begin{theorem}\label{Gl4 result}
Using the notation introduced for Proposition \ref{main prop}, an irreducible admissible representation of ${\rm GL}_{4}(F)$ with a symplectic period is a product of $\chi_{i}U(\delta_{i},2n_{i})$ where $\chi_{i}$'s are (not necessarily unitary) characters of $F^{\times}$.
\end{theorem}
Next we state the corresponding theorem for ${\rm GL}_{6}(F)$.
\begin{theorem}\label{main theorem}
Using the notation introduced for Proposition \ref{main prop}, an irreducible admissible representation of ${\rm GL}_{6}(F)$ with a symplectic period is either a product of $\chi_{i}U(\delta_{i},2n_{i})$ ($\chi_{i}$'s are not necessarily unitary) or is a twist of $Z([1,\nu],[\nu,\nu^{4}])$, or its dual.
\end{theorem}

A few words about the proofs. It is a consequence of the uniqueness of the Klyachko models that irreducible cuspidal representations (which are generic) cannot have a symplectic period. Since any non-supercuspidal irreducible representation is a quotient of a representation of the form ${\rm ind}^{{\rm GL}_{2n}}_{P_{k,2n-k}} (\rho \otimes \tau) ,\rho \in Irr ({\rm GL}_{k}(F)) , \tau \in Irr ({\rm GL}_{2n-k}(F))$ it is enough to study the problem for representations of these types. For ${\rm GL}_{4}(F)$ and ${\rm GL}_{6}(F)$, this reduces the problem to the analysis of representations of the type $\pi_{1}\times\pi_{2}$ and $\pi_{1}\times\pi_{2}\times\pi_{3}$ where each $\pi_{i}$ is an irreducible representation of ${\rm GL}_{2}(F)$. For the ${\rm GL}_{4}(F)$ case, using Mackey theory we obtain an exhaustive list of representations (not necessarily irreducible). Then we study every possible quotient to obtain a complete list of irreducible ${\rm Sp}_{2}(F)$-distinguished representations of ${\rm GL}_{4}(F)$. In the ${\rm GL}_{6}(F)$ case, we first reduce the problem to the case when none of the $\pi_{i}$'s are cuspidal. Next we reduce the problem to the case when at most one of the $\pi_{i}$'s is an irreducible principal series. Then we do a case-by-case analysis (for each $\pi_{i}$ to be one of the three types of irreducible representations of ${\rm GL}_{2}(F)$ - a character, an irreducible principal series or a twist of the Steinberg with at most one of them an irreducible principal series), analyzing all possible subquotients for symplectic periods. A common way of showing that an irreducible subquotient is not $H$-distinguished, specially in the ${\rm GL}_{6}(F)$ case, is to express it as a quotient of a representation, which is then shown not to have a symplectic period using Mackey theory. 

A word on the organization of the paper. Section \ref{Notation and Preliminaries} consists of the notation and preliminary notions which has been used in the paper. The orbit structures and Mackey theory is done in detail in Section \ref{Orbit structures and mackey theory}. We analyze the representations of the form $\pi_{1}\times\pi_{2}$ and obtain the theorem for ${\rm GL}_{4}(F)$ in Section \ref{analysis Gl4}. In Section \ref{Analysis of}, we analyze the representations of the form $\pi_{1}\times\pi_{2}\times\pi_{3}$, collecting all the irreducible ${\rm Sp}_{3}(F)$-distinguished subquotients. Using the analysis, we obtain the theorem for ${\rm GL}_{6}(F)$. In Section \ref{conjectures} we make a few conjectures for the general case based on the available examples.

{\bf Acknowledgements:} This work is a part of author's thesis. He wishes to thank his advisor Dipendra Prasad for suggesting the problem studied in this work and guidance. Further, he wishes to thank Omer Offen, Eitan Sayag and Steven Spallone for their interest and suggestions to improve the paper. 

\section {Notation and Preliminaries} \setcounter{equation}{0}\label{Notation and Preliminaries}
\subsection{Notation}\label{Notation}

All along in this paper, the field $F$ will denote a non-archimedean local field of characteristic 0. 

Following the notation of \cite {BZ1}, we denote the set of all smooth representations of an $l$-group $G$ by $Alg (G)$ and the subset of all irreducible admissible representations by $Irr (G)$. If $\pi \in Alg (G)$, we denote by $\tilde{\pi}$, its contragredient.  

Any character of ${\rm GL}_{n}(F)$ can be thought of as a character of $F^{\times}$ via the determinant map. Given a character $\chi$ of $F^{\times}$ and a smooth representation $\pi$ of ${\rm GL}_{n}(F)$ we will denote the twist of $\pi$ by $\chi$ simply by $\chi \pi$, $\chi \pi (g):= \chi(\det (g))\pi(g)$. Unless otherwise mentioned, $St_{n}$ and $1_{n}$ will be used to denote the Steinberg and the trivial character of ${\rm GL}_{n}(F)$ respectively. The norm character $\nu(g):=|\det(g)|$ will be denoted by $\nu$. 

Let $P_{n_{1} ,\cdots,n_{r}}$ be the group of block upper triangular matrices corresponding to the tuple $(n_{1} ,...,n_{r})$. Let $N_{n_{1} ,\cdots,n_{r}}$ denote its unipotent radical. Let $\delta_{P_{n_{1} ,...,n_{r}}}$ denote the modular function of the group $P_{n_{1} ,\cdots,n_{r}}$. Since a parabolic normalizes its unipotent radical, this defines a character of $P_{n_{1} ,\cdots,n_{r}}$ (the module of the automorphism  $n\to pnp^{-1}$ of $N_{n_{1} ,\cdots,n_{r}}$ for $p \in P_{n_{1} ,\cdots,n_{r}}$). Call this character $\delta_{N_{n_{1} ,...,n_{r}}}$. Then we have $\delta_{N_{n_{1} ,...,n_{r}}}=\delta_{P_{n_{1} ,...,n_{r}}}$. For an element $p \in P_{n_{1} ,\cdots,n_{r}}$, with its Levi part equal to $\textrm{diag}(g_{1},...,g_{r})$, 
\begin{align} \label{modular}
\delta_{P_{n_{1} ,...,n_{r}}}(p)=& |\det(g_{1})|^{n_{2}+\cdots+n_{r}}|\det(g_{2})|^{-n_{1}+n_{3}+\cdots+n_{r}}\cdots|\det(g_{r})|^{-n_{1}\cdots-n_{r-1}}. 
\end{align}

The induced representation of $(\sigma,H,W)\in Alg (H)$ to $G$ is the following space of locally constant functions 
$$\textrm{Ind}^{G}_{H}\sigma=\{f:G\to W\ |\ f(hg)=\delta_{H}^{1/2}\delta_{G}^{-1/2}\sigma(h)f(g),\ \forall h\in H,g\in G\}$$
where $\delta_{G}$ and $\delta_{H}$ are the modular functions of $G$ and $H$ respectively. $G$ acts on the space by right action. Compact induction from $H$ to $G$ is denoted by ${\rm ind}^{G}_{H}\sigma$ and is the subspace of $\textrm{Ind}^{G}_{H}\sigma$ consisting of functions compactly supported mod $H$. Occasionally we will use non-normalized induction (see Remark 2.22 of \cite {BZ1} for the definition), although unless otherwise mentioned induction is always normalized. Given representations $\rho_{i} \in Irr ({\rm GL}_{n_{i}}(F) )(i=1,...,r)$, extend $\rho_{1} \otimes \cdots\otimes \rho_{r}$ to $P_{n_{1},...,n_{r}}$ so that it is trivial on $N_{n_{1},...,n_{r}}$. We denote by $\rho_{1} \times \cdots\times \rho_{r}$ the representation ${\rm ind}^{{\rm GL}_{n}}_{P_{n_{1},...,n_{r}}}\ \rho_{1} \otimes \cdots\otimes \rho_{r}$.

The Jacquet functor with respect to a unipotent subgroup $N$ is denoted by $r_{N}$ and is always normalized. 

If $\pi\in Irr(\textrm{GL}_{n}(F))$, then there exists a partition of $n$ and a multiset of cuspidal representations $\{\rho_{1},\cdots,\rho_{r}\}$ corresponding to it such that $\pi$ can be embedded in $\rho_{1}\times \cdots\times \rho_{r}$. This multiset is uniquely determined by $\pi$ and called its cuspidal support. For the purpose of this paper, for a smooth representation of finite length define it to be the union (as a set) of all the supports of its irreducible subquotients. 

\subsection{Preliminaries on segments}\label{Preliminaries on segments}
We briefly recall the notation and the basic definition of ``segments'' as introduced by Zelevinsky (in \cite{Z}). Given a cuspidal representation $\rho$ of ${\rm GL}_{m}(F)$ define a segment to be a set $\Delta$ of the form $\{ \rho, \rho\nu,...,\rho\nu^{k-1} \}(k>0)$. Given a segment $\Delta=\{ \rho, \rho\nu,...,\rho\nu^{k-1} \}$ (also denoted by $[\rho , \rho\nu^{k-1}]$), the unique irreducible submodule and the unique irreducible quotient of $\rho \times \cdots\times \rho \nu^{k-1}$ is denoted by $Z(\Delta)$ and $Q(\Delta)$ respectively. 

For $\Delta_{1}=[\rho_{1},\nu^{k_{1}-1}\rho_{1}], \Delta_{2}=[\rho_{2},\nu^{k_{2}-1}\rho_{2}]$, we say that $\Delta_{1}$ and $\Delta_{2}$ are linked if $\Delta_{1}\nsubseteq \Delta_{2}$, $\Delta_{2}\nsubseteq \Delta_{1}$ and $\Delta_{1}\cup \Delta_{2}$ is also a segment. If $\Delta_{1}$ and $\Delta_{2}$ are linked and $\Delta_{1}\cap \Delta_{2}=\phi$, then we say that $\Delta_{1}$ and $\Delta_{2}$ are juxtaposed. If $\Delta_{1}$ and $\Delta_{2}$ are linked and $\rho_{2}=\nu^{k}\rho_{1}$ where $k>0$ then we say that $\Delta_{1}$ precedes $\Delta_{2}$. Given a multiset $\mathfrak a = \{ \Delta_{1},...,\Delta_{r}\}$ of segments, let $$\pi(\mathfrak a):= Z(\Delta_{1}) \times\cdots \times Z(\Delta_{r}).$$ If $\Delta_{i}$ does not precede $\Delta_{j}$ for any $i<j$, $\pi (\mathfrak a)$ is known to have a unique irreducible submodule which will be denoted by $Z(\Delta_{1},...,\Delta_{r})$. By Theorem 6.1 of \cite{Z}, this submodule is independent of the ordering of the segments as long as the ``does not precede'' condition is satisfied. Hence we simply denote it by $Z(\mathfrak a)$. In this situation, a similar statement holds for quotients as well and the unique irreducible quotient of $Q(\Delta_{1}) \times\cdots \times Q(\Delta_{r})$ is denoted by $Q(\mathfrak a)$. For e.g., the trivial character $1_{n}$ of ${\rm GL}_{n}(F)$ is $Z([\nu^{-(n-1)/2},\nu^{(n-1)/2}])$ while $St_{n}$ is $Q([\nu^{-(n-1)/2},\nu^{(n-1)/2}])$.  

We say a multiset $\mathfrak a=\{ \Delta_{1},...,\Delta_{r}\}$ is on the cuspidal line of $\rho$, where $\rho$ is a cuspidal representation of some ${\rm GL}_{n}(F)$, if $\Delta_{i}\subset \{\nu^{k}\rho\}_{k\in \mathbb Z},\ \forall i$.

\subsection{Preliminaries on ${\rm GL}_{n}(F)$ and symplectic periods}\label{Preliminaries on general}
In this subsection we collect a few basic results on ${\rm GL}_{n}(F)$ and symplectic periods which are used in the sequel.
The following result is used to calculate explicitly the quotients and the submodules in quite a few cases in the proofs of the main theorems.
\begin{theorem}[{\bf Zelevinsky, \cite {Z}}]\label{product of sub}
Let $\Delta_{1}$ and $\Delta_{2}$ be two segments. If $\Delta_{1}$ and $\Delta_{2}$ are linked, put $\Delta_{3}=\Delta_{1} \cup \Delta_{2}$ and  $\Delta_{4}=\Delta_{1} \cap \Delta_{2}$ . The representation $\pi = Z(\Delta_{1}) \times Z(\Delta_{2})$ is irreducible if and only if  $\Delta_{1}$ and $\Delta_{2}$ are not linked. If  $\Delta_{1}$ and $\Delta_{2}$ are linked then $\pi$ has length 2. If $\Delta_{2}$ precedes $\Delta_{1}$ then $\pi$ has unique irreducible submodule $Z(\Delta_{1},\Delta_{2})$ and unique irreducible quotient $Z(\Delta_{3}) \times Z(\Delta_{4})$. If $\Delta_{1}$ precedes $\Delta_{2}$ then $\pi$ has unique irreducible submodule $Z(\Delta_{3}) \times Z(\Delta_{4})$ and unique irreducible quotient $Z(\Delta_{1},\Delta_{2})$.
\end {theorem} 
Using Zelevinsky involution and Rodier's theorem that $Q(\Delta_{1},\Delta_{2})$ is taken to $Z(\Delta_{1},\Delta_{2})$ we have a quotient version of this lemma.
\begin{theorem}\label{product of quotient}
Let $\Delta_{1}$ and $\Delta_{2}$ be two segments. If $\Delta_{1}$ and $\Delta_{2}$ are linked, put $\Delta_{3}=\Delta_{1} \cup \Delta_{2}$ and  $\Delta_{4}=\Delta_{1} \cap \Delta_{2}$ . The representation $\pi = Q(\Delta_{1}) \times Q(\Delta_{2})$ is irreducible if and only if  $\Delta_{1}$ and $\Delta_{2}$ are not linked. If  $\Delta_{1}$ and $\Delta_{2}$ are linked then $\pi$ has length 2. If $\Delta_{2}$ precedes $\Delta_{1}$ then $\pi$ has the unique irreducible submodule $Q(\Delta_{3}) \times Q(\Delta_{4})$. If $\Delta_{1}$ precedes $\Delta_{2}$ then $\pi$ has the unique irreducible quotient $Q(\Delta_{3}) \times Q(\Delta_{4})$.
\end {theorem} 

\begin{lemma}[{\bf Duality Lemma, \cite {Cas}}]\label{duality lemma}
Let $\eta=\rho_{1}\times\cdots\times\rho_{m}$ where $\rho_{i}\in\ Irr({\rm GL}_{i}(F))$. Define $\breve{\eta}=\tilde{\rho_{m}}\times\cdots\times\tilde{\rho_{1}}$. If $\pi$ is an irreducible quotient of $\eta$ then $\tilde{\pi}$ is an irreducible quotient of
$\breve{\eta}$.
\end{lemma}

Let ${\rm Ext}_{G}^{1}(..,\mathbb C)$ be the derived group of the ${\rm Hom}_{G}(..,\mathbb C)$ functor (for details see, \cite{Prasad1} and \cite{Prasad2}). Then we have the following lemma.
\begin{lemma}\label{ext1}
Let $H={\rm Sp}_{n}(F)$. Then ${\rm Ext}_{H}^{1}(\mathbb C, \mathbb C)$ is trivial.
\end{lemma}
\begin{proof}
An element of ${\rm Ext}_{H}^{1}(\mathbb C, \mathbb C)$ corresponds to an exact sequence 
$$0\rightarrow \mathbb C \overset{i}\rightarrow V \overset{j} \rightarrow\mathbb C \rightarrow 0 $$
of $H$-modules, or equivalently a homomorphism from $H$ to the group of upper triangular unipotent subgroup of ${\rm GL}_{2}(\mathbb C)$. But as $H$ has no abelian quotients, there are no such non-trivial maps and we have the lemma.  
\end{proof}

A result which plays a very important role in the proofs is the fact that irreducible generic representations have no symplectic period. The following more general result is due to Offen and Sayag.
\begin {theorem}[{\bf Offen, Sayag}, \cite{Offen-Sayag3}]\label{strong uniqueness}
Let $\pi \in Irr ({\rm GL}_{n}(F))$. If $\pi$ embeds in a Klyachko model, it does so in a unique Klyachko model and with multiplicity at most one.
\end{theorem}

\section{Orbit Structures and Mackey theory}\setcounter{equation}{0}\label{Orbit structures and mackey theory}

Let $X$ be a subspace of a symplectic space $(V,\langle,\rangle)$ of dimension $2n$. Let
$$ X^{\perp}=\{y\in V|\ \langle y,x\rangle=0\ \ \forall x\in X\}.$$Define Rad $X=X\cap X^{\perp}$. Note that $X/ {\rm Rad} \ X$ inherits the symplectic structure of $V$, becomes a non-degenerate symplectic space and hence has even dimension. Then we have the following lemma which is a variant of the classical theorem of Witt for quadratic forms.

\begin{lemma}[{\bf Witt}]\label{witt}
a) Let $X_{1}, X_{2}$ be subspaces of $V$ of same dimension. Then there exists a symplectic automorphism $\phi$ of $V$, taking $X_{1}$ to $X_{2}$ iff dim Rad $X_{1}$ = dim Rad $X_{2}$.

b) Let $X_{1},X_{2}$ be two subspaces of $V$ and $\phi : X_{1} \mapsto X_{2}$ be a symplectic isomorphism. Then $\phi$ extends to a symplectic automorphism of $V$. 
\end{lemma}

It follows from this lemma that if $X$ is a $k$-dimensional subspace of $V$, and $P_{X}$ is the parabolic subgroup of ${\rm GL}(V)$ consisting of automorphisms of $V$ leaving $X$ invariant, then ${\rm Sp}(V)\setminus {\rm GL}(V)/P_{X}$ is in bijective correspondence with integers $i$, $0\leq i\leq {\rm dim} \ X$ such that ${\rm dim} \ X-i$ is even. To get a set of representatives for these double cosets, let   
$$\{ e_{1}, e_{2},..., e_{n}, f_{1},f_{2},...,f_{n} \}$$be the standard symplectic basis of $V$, i.e. $\langle e_{i},f_{j}\rangle=\delta_{ij}$. Define,
$$Y_{r}:=\langle e_{1},...,e_{r}\rangle$$
$$Y_{r}^{\vee}:=\langle f_{1},...,f_{r}\rangle$$
$$S_{k,r}:=\langle e_{r+1},...,e_{(k+r)/2},f_{r+1},...,f_{(k+r)/2}\rangle$$
$$T_{k,r}:=\langle e_{\frac{k+r}{2}+1},...,e_{n},f_{\frac{k+r}{2}+1},...,f_{n}\rangle$$
$$X_{k,r}:=Y_{r}+S_{k,r}.$$
Note that ${\rm GL}(V)/P_{X}$ is the set of all $k$-dimensional subspaces of $V$ on which ${\rm Sp}(V)$ acts in a natural way. Therefore ${\rm Sp}(V)\setminus {\rm GL}(V)/P_{X}$ is represented by a certain set of $k$-dimensional subspaces of $V$ which can be taken to be the spaces $X_{k,r}$ with $0\leq r\leq k$ such that $k-r$ is even. 

Since $\textrm{dim}\ X= \textrm{dim}\ X_{k,r}$ there exists an automorphism $g\in \textrm{GL}(V)$ taking $X$ to $X_{k,r}$. This automorphism gives an isomorphism from $P_{X}$ to $P_{X_{k,r}}$. Using this isomorphism a representation of $P_{X}$ can be considered to be a representation of $P_{X_{k,r}}$. By Mackey theory, the restriction of the representation ${\rm Ind}^{{\rm GL}(V)}_{P_{X}}(\sigma)$ to ${\rm Sp}(V)$ is obtained by glueing the representations,
$${\rm ind}^{{\rm Sp}(V)}_{({\rm Sp}(V) \cap P_{X_{k,r}})}(\delta_{P_{X}}^{1/2}\sigma|_{({\rm Sp}(V) \cap P_{X_{k,r}})})$$where the induction is non-normalized. Note that the isomorphism of $P_{X}$ with $P_{X_{k,r}}$ takes the unipotent radical of $P_{X}$ to the unipotent radical of $P_{X_{k,r}}$ and hence the representation of $P_{X_{k,r}}$ so obtained is of the same kind that appears in parabolic induction. We remark that this is a special case for maximal parabolics of Proposition 3 of \cite{Offen}.  
  
For an isotropic subspace $Y$ of $V$, the subgroup $Q_{Y}$ of ${\rm Sp}(V)$ stabilizing $Y$ is a parabolic subgroup of ${\rm Sp}(V)$, with Levi decomposition
$$Q_{Y}=\big({\rm GL}(Y)\times {\rm Sp}(Y^{\perp}/Y)\big)\ltimes U$$where $U$ is the subgroup of ${\rm Sp}(V)$ preserving $Y\subset Y^{\perp}$ and acting trivially on $Y$, $Y^{\perp}/Y$ and $V/Y^{\perp}$. 

We fix a symplectic basis of $V$ and identify the group of linear transformations with the corresponding group of matrices although we would like to emphasize that the following proposition and its corollary is independent of the choice of the basis.
\begin{proposition}\label{explicit description of stab}
The subgroup $H_{k,r}$ of ${\rm Sp}(V)$ stabilizing the subspace $X_{k,r}$ of $V$ is$$H_{k,r}=\big({\rm GL}(Y_{r})\times {\rm Sp}(S_{k,r})\times {\rm Sp}(T_{k,r})\big).U_{k,r}$$where $U_{k,r}$ is the unipotent group inside ${\rm Sp}(V)$ consisting of automorphisms of $V$ of the form
$$\left(\begin{array}{cccc}
I_{r}&A&B&C \\
0&I_{k-r}&0&A^{'} \\
0&0&I_{2n-(k+r)}&B^{'}\\
0&0&0&I_{r}  
\end{array}\right)$$with $A\in {\rm Hom}(S_{k,r},Y_{r})$, $B\in {\rm Hom}(T_{k,r},Y_{r})$ and $A^{'}\in {\rm Hom}(Y_{r}^{\vee},S_{k,r})$, the adjoint of $A$, and $B^{'}\in {\rm Hom}(Y_{r}^{\vee},T_{k,r})$, the adjoint of $B$; the matrix $C\in{\rm Hom}(Y_{r}^{\vee},Y_{r})$ is symmetric.  
\end{proposition}
\begin{proof}
Note that $H_{k,r}$ is nothing but the symplectic automorphisms of $V$ preserving the flag:
$$0\subset Y_{r}=X_{k,r}\cap X_{k,r}^{\perp}\subset X_{k,r}\subset X_{k,r}+X_{k,r}^{\perp}=X_{k,r}+T_{k,r}=Y_{r}^{\perp}\subset V.$$Hence $H_{k,r}$ acts on the successive quotients of this filtration, giving rise to a surjective homomorphism to ${\rm GL}(Y_{r})\times {\rm Sp}(S_{k,r})\times {\rm Sp}(T_{k,r})$ with kernel $U_{k,r}$ consisting of the subgroup of ${\rm Sp}(V)$ preserving the flag and acting trivially on successive quotients. Clearly $U_{k,r}$ acts trivially on the isotropic subspace $Y_{r}$, $Y^{\perp}_{r}$ and $Y_{r}^{\perp}/Y_{r}=S_{k,r}+T_{k,r}$. The well-known knowledge of the structure of the parabolic in ${\rm Sp}(V)$ defined by $Y_{r}$ proves the assertion of the proposition.      
\end{proof}

\begin{corr}\label{modular characters}
1) The modular character $\delta_{k,r}$ of the group $H_{k,r}$, is $\newline$$\delta_{k,r}({\rm diag}(g,h_{1},h_{2},\ ^{t}g^{-1}))=|\det (g)|^{r+a+b+1}$ where $r$=dim $Y_{r}$, $a$=dim $S_{k,r}=k-r$, $b$=dim $T_{k,r}=2n-(k+r)$ and $g\in {\rm GL}(Y_{r})$; $\newline$
2) By eqn.(\ref{modular}), for $P=P_{(r+a,b+r)}$, $\delta_{P}({\rm diag}(g,h_{1},h_{2},\ ^{t}g^{-1}))=|\det (g)|^{2r+a+b}$. Thus $\frac{\delta_{P}^{1/2}}{\delta_{k,r}}({\rm diag}(g,h_{1},h_{2},\ ^{t}g^{-1}))=|\det (g)|^{-1-(a+b)/2}=|\det (g)|^{-(n-r+1)}$. 
\end{corr}
Define $M$ to be the group ${\rm GL}(Y_{r})\times {\rm Sp}(S_{k,r})\times {\rm Sp}(T_{k,r})$ and identify it with ${\rm GL}_{r}(F)\times {\rm Sp}_{(k-r)/2}(F)\times {\rm Sp}_{(2n-k-r)/2}(F)$ via the fixed basis. Call $H$ the group ${\rm Sp}_{n}(F)$ defined with respect to this symplectic basis. Further let $N=N_{1}\times N_{2}$ (where $N_{1}$, $N_{2}$ are the unipotent subgroups of ${\rm GL}_{k}(F)$ and ${\rm GL}_{2n-k}(F)$ corresponding to the partition $(r,k-r)$ and $(2n-k-r,r)$ respectively). Let $\sigma_{1}\in Irr({\rm GL}_{k}(F))$ and $\sigma_{2}\in Irr({\rm GL}_{2n-k}(F))$. Call $\sigma$ the representation of $P=P_{(k,2n-k)}$ obtained by extending $\sigma_{1}\otimes \sigma_{2}$ to $P$ in the usual way. 

By Frobenius reciprocity and Corollary \ref{modular characters}, we get $${\rm Hom}_{H}\big({\rm ind}^{H}_{H_{k,r}}(\delta_{P}^{1/2}\sigma|_{H_{k,r}}),\mathbb C\big) ={\rm Hom}_{MU_{k,r}}\big(\nu^{-(n-r+1)}\sigma_{1}\otimes\sigma_{2},\mathbb C\big).$$ Clearly,
$${\rm Hom}_{M.U_{k,r}}\big(\nu^{-(n-r+1)}\sigma_{1}\otimes\sigma_{2},\mathbb C\big)={\rm Hom}_{MN}\big(\nu^{-(n-r+1)}\sigma_{1}\otimes\sigma_{2},\mathbb C\big).$$
Since the normalized Jacquet functor is left adjoint to normalized induction (cf. Proposition 1.9 (b) of \cite{BZ2}) we have,$${\rm Hom}_{MN}\big(\nu^{-(n-r+1)}\sigma_{1}\otimes\sigma_{2},\mathbb C\big)={\rm Hom}_{M}\big(r_{N}(\nu^{-(n-r+1)}\sigma_{1}\otimes\sigma_{2}),\delta_{N}^{-1/2}\big)$$
$$={\rm Hom}_{M}\big(\nu^{-(n-r+1)}\delta_{N_{1}}^{1/2}r_{N_{1}}(\sigma_{1})\otimes \delta_{N_{2}}^{1/2}r_{N_{2}}(\sigma_{2}),\mathbb C\big).\ $$
Now (by eqn.(\ref{modular})) for $\det A=\det B=1$, $$\delta_{N_{1}}\left(\begin{array}{cc}
g&* \\
0&A  
\end{array}\right)=|\det(g)|^{(k-r)}, 
\delta_{N_{2}}\left(\begin{array}{cc}
B&* \\
0&\ ^{t}g^{-1}  
\end{array}\right)=|\det(g)|^{2n-(k+r)}.$$Define $\alpha$ to be the character of $M$ such that $\alpha({\rm diag}(g,h_{1},h_{2},\ ^{t}g^{-1}))=\nu^{-1}(g)$. Plugging in the value of the delta functions we get, 
\begin{align}\label{general case eqn}
{\rm Hom}_{H}\big({\rm ind}^{H}_{H_{k,r}}(\delta_{P}^{1/2}\sigma|_{H_{k,r}}),\mathbb C\big)={\rm Hom}_{M}\big(\alpha (r_{N_{1}}(\sigma_{1})\otimes r_{N_{2}}(\sigma_{2})),\mathbb C\big).
\end{align}
From this we have the following lemma for ${\rm GL}_{2n}(F)$.

\begin{lemma}\label{irr}
Let $\pi_{i}=Z(\Delta_{1}^{i},...,\Delta_{k_{i}}^{i})\in Irr({\rm GL}_{n_{i}}(F))$ for $i=1,...,s$ be such that the following conditions are satisfied:
$\newline$1) For $i\neq j$, $\Delta_{m_{i}}^{i},\Delta_{m_{j}}^{j}$ are disjoint and not linked ($\forall \ m_{i}=1,...,k_{i};\ \forall\ m_{j}=1,...,k_{j}$). $\newline$2) $\Sigma_{i=1}^{s}n_{i}$ is even and $\pi:=\pi_{1}\times\cdots\times\pi_{s}$ has a symplectic period.
$\newline$Then each $n_{i}$ is even and every $\pi_{i}$ has a symplectic period.    
\end{lemma}
\begin{proof}
Observe that condition 1) forces $\pi$ to be irreducible (by Proposition 8.5 of \cite{Z}). Thus it is enough to prove the lemma for $s=2$. 

Let $\pi_{1}\in Irr({\rm GL}_{n_{1}}(F))$ and $\pi_{2}\in Irr({\rm GL}_{n_{2}}(F))$. Now, since $r_{N_{1}}(\pi_{1})\in Alg({\rm GL}_{r}(F)\times {\rm GL}_{n_{1}-r}(F))$ and the functor $r_{N_{1}}$ takes finite length representations into ones of finite length (Proposition 1.4 of \cite{Z}), up to semi-simplification it is of the form $\Sigma_{i=1}^{t_{1}}\pi_{1i}\otimes\tau_{1i}$ for some $t_{1}>0$ where $\pi_{1i}\in Irr({\rm GL}_{r}(F))$ and $\tau_{1i}\in Irr({\rm GL}_{n_{1}-r}(F))$ for all $i=1,..,t_{1}$. Similarly, up to semi-simplification, $r_{N_{2}}(\pi_{2})$ is equal to $\Sigma_{j=1}^{t_{2}}\tau_{2j}\otimes\pi_{2j} $ where $\tau_{2j}\in Irr({\rm GL}_{n_{2}-r}(F))$, $\pi_{2j}\in Irr({\rm GL}_{r}(F))$. 

We claim that for any $\theta\in Irr ({\rm GL}_{m}(F))$, the cuspidal support (see Section \ref{Notation} for the definition) of $r_{N}(\theta)$ is always a subset (as a set) of the cuspidal support of $\theta$. Assume $\theta=Z(\Delta_{1},...,\Delta_{l})$. The statement of geometrical lemma (Lemma 2.12 of \cite {BZ2}) applied to $r_{N}(Z(\Delta_{1})\times...\times Z(\Delta_{l}))$ along with the observation that $r_{N}(\theta)$ is a submodule of it gives the claim.   

The above claim and condition 1) of the lemma implies that ${\rm Hom}_{{\rm GL}_{r}}(\nu^{-1}\pi_{1i}\otimes\tilde{\pi_{2j}},\mathbb C)=0$ for every pair $i,j$. By \ref{general case eqn} and the realization of contragredient representations due to Gelfand and Kazhdan (cf. Theorem 7.3 of \cite{BZ1}), this implies$${\rm Hom}_{H}\big({\rm ind}^{H}_{H_{n_{1},r}}(\delta_{P_{n_{1},n_{2}}}^{1/2}(\pi_{1}\otimes\pi_{2})|_{H_{n_{1},r}}),\mathbb C\big)=0 $$unless $r=0$. This along with condition 2) forces $n_{1},n_{2}$ to be even and $\pi_{1},\pi_{2}$ both to have symplectic periods.   
\end{proof}
Before we state the next lemma, for the convenience of the reader let us recall that if $\Delta=[\rho,\nu^{k-1}\rho]$, then $Z(\Delta)\cong U(\nu^{\frac{k-1}{2}}\rho,k)$.
\begin{lemma}\label{existence of symplectic period lemma}
Let $\Delta_{1}$ and $\Delta_{2}$ be two segments of even lengths such that their intersection is of odd length. Then the representation $\theta=Z(\Delta_{1},\Delta_{2})$ has a symplectic period.
\end{lemma}
\begin{proof}
Let if possible ${\rm Hom}_{H}\big(\theta,\mathbb C\big)=0$. Define, the segments $\Delta_{3}=\Delta_{1}\cup\Delta_{2}$ and $\Delta_{4}=\Delta_{1}\cap\Delta_{2}$. Without loss of generality assume $\Delta_{1}$ precedes $\Delta_{2}$. By Theorem \ref{product of sub}, $\theta$ sits inside the following exact sequence of ${\rm GL}_{2n}(F)$ modules:
$$0 \rightarrow \theta \rightarrow Z(\Delta_{2})\times Z(\Delta_{1})\rightarrow Z(\Delta_{3})\times Z(\Delta_{4})\rightarrow 0.$$Observe that $\Delta_{3}$ and $\Delta_{4}$ are segments of odd length. So, $Z(\Delta_{3})\times Z(\Delta_{4})$ has a mixed Klyachko model by Theorem 3.7 of \cite{Offen-Sayag2} and hence by Theorem \ref{strong uniqueness}, it is not $H$-distinguished. Since ${\rm Hom}_{H}\big(Z(\Delta_{2})\times Z(\Delta_{1}),\mathbb C\big)=0$ if ${\rm Hom}_{H}\big(Z(\Delta_{3})\times Z(\Delta_{4}),\mathbb C\big)=0$ and ${\rm Hom}_{H}\big(\theta,\mathbb C\big)=0$, this contradicts Proposition \ref{main prop}.
\end{proof}
\begin{lemma}\label{juxt no symplectic period lemma}
Let $\Delta_{1}$ and $\Delta_{2}$ be two juxtaposed segments of even lengths in the cuspidal line of $1_{1}$ (the trivial representation of ${\rm GL}_{1}(F)$). Then the representation $\theta=Z(\Delta_{1},\Delta_{2})$ doesn't have a symplectic period. 
\end{lemma}
\begin{proof}
Define $\Delta_{3}=\Delta_{1}\cup\Delta_{2}$ and let $2n$ be its length. In fact, twisting it by an appropriate power of $\nu$, without loss of generality we can take $\Delta_{3}$ to be $[\nu^{-\frac{2n-1}{2}},\nu^{\frac{2n-1}{2}}]$ and hence $Z(\Delta_{3})=1$. Let $$\Delta_{1}=[\nu^{-\frac{2n-1}{2}},\nu^{\frac{a}{2}}] \ \textrm{and} \ \Delta_{2}=[\nu^{\frac{b}{2}},\nu^{\frac{2n-1}{2}}].$$Let $k=\frac{2n-1}{2}-\frac{b}{2}+1$, the length of $\Delta_{2}$. Now assume $k\leq n$.

Let $\mu_{1}=Z(\Delta_{2})$ and $\mu_{2}=Z(\Delta_{1})$. Let us first calculate ${\rm Hom}_{H}\big(\mu_{1}\times\mu_{2},\mathbb C\big)$. By eqn. \ref{general case eqn}, for $r\neq 0,{\rm Hom}_{H}\big({\rm ind}^{H}_{H_{k,r}}(\delta_{P}^{1/2}\mu_{1} \otimes \mu_{2}|_{H_{k,r}}),\mathbb C\big)\cong$ 
$${\rm Hom}_{{\rm GL}_{r}(F)\times {\rm Sp}_{\frac{(k-r)}{2}}(F)\times {\rm Sp}_{n-\frac{(k+r)}{2}}(F)}(\nu^{n-\frac{2k-r}{2}-1}\otimes\nu^{n-1-\frac{k-r}{2}}\otimes\nu^{-\frac{(k+r)}{2}}\otimes\nu^{n-k-\frac{r}{2}},\mathbb C)$$
where ${\rm GL}_{r}(F)$ acts on the last term via the contragredient. Now, consider $${\rm Hom}_{{\rm GL}_{r}(F)}(\nu^{n-\frac{2k-r}{2}-1}\otimes\nu^{-(n-k-\frac{r}{2})},\mathbb C).$$ This is non-zero only if $n-\frac{2k-r}{2}-1=n-k-\frac{r}{2}$ which is impossible since $k$ is even by the hypothesis of the lemma. So, by Theorem 7.3 of \cite{BZ1},
$${\rm Hom}_{H}\big({\rm ind}^{H}_{H_{k,r}}(\delta_{P}^{1/2}\mu_{1} \otimes \mu_{2}|_{H_{k,r}}),\mathbb C\big)=0 \ \textrm{if}\ r\neq 0.$$
On the other hand, if $r=0$ 
$${\rm Hom}_{H}\big({\rm ind}^{H}_{H_{k,0}}(\delta_{P}^{1/2}\mu_{1} \otimes \mu_{2}|_{H_{k,0}}),\mathbb C\big)={\rm Hom}_{{\rm Sp}_{\frac{k}{2}}(F)}\big(\mu_{1},\mathbb C\big)\otimes {\rm Hom}_{{\rm Sp}_{n-\frac{k}{2}}(F)}\big(\mu_{2},\mathbb C\big)=\mathbb C.$$Hence ${\rm Hom}_{H}\big(\mu_{1}\times\mu_{2},\mathbb C\big)$ is at most one dimensional. 
Now, we have the following exact sequence of ${\rm GL}_{2n}(F)$ modules (and hence of ${\rm Sp}_{n}(F)$ modules):
$$0 \rightarrow Z(\Delta_{1},\Delta_{2})\overset{i}\rightarrow Z(\Delta_{2})\times Z(\Delta_{1})\overset{j}\rightarrow \mathbb C \rightarrow 0.$$
Applying the functor ${\rm Hom}_{{\rm Sp}_{n}(F)}\big(..,\mathbb C\big)$ to it we obtain the following long exact sequence:
$$0 \rightarrow {\rm Hom}_{{\rm Sp}_{n}(F)}\big(\mathbb C, \mathbb C\big) \overset{j^{*}}\rightarrow {\rm Hom}_{{\rm Sp}_{n}(F)}\big(Z(\Delta_{2})\times Z(\Delta_{1}),\mathbb C\big) \ \ \ \ \ \ \ \ \ \ \ \ \ $$
$$\ \ \ \ \ \ \ \ \ \ \ \ \ \ \ \ \ \ \ \overset{i^{*}}\rightarrow {\rm Hom}_{{\rm Sp}_{n}(F)}\big(Z(\Delta_{1},\Delta_{2}), \mathbb C\big)\overset{r^{*}}\rightarrow {\rm Ext}_{{\rm Sp}_{n}(F)}^{1}\big(\mathbb C, \mathbb C\big)\rightarrow\ ...$$Observing that ${\rm Ext}_{{\rm Sp}_{n}(F)}^{1}\big(\mathbb C, \mathbb C\big)=0$ (see Lemma \ref{ext1}) we get the following short exact sequence:
$$0 \rightarrow \mathbb C\overset{j^{*}}\rightarrow {\rm Hom}_{{\rm Sp}_{n}(F)}\big(Z(\Delta_{2})\times Z(\Delta_{1}),\mathbb C\big)\overset{i^{*}}\rightarrow {\rm Hom}_{{\rm Sp}_{n}(F)}\big(Z(\Delta_{1},\Delta_{2}),\mathbb C\big) \rightarrow 0.$$
Since $j^{*}$ is injective, ${\rm Im}(j^{*})=\mathbb C$. By exactness, ${\rm Ker}(i^{*})=\mathbb C$ as well. Since ${\rm Hom}_{{\rm Sp}_{n}(F)}\big(Z(\Delta_{2})\times Z(\Delta_{1}),\mathbb C\big)$ was shown to be at most one dimensional, it is equal to ${\rm Ker}(i^{*})$. Thus ${\rm Im}(i^{*})=0$. But again by exactness, $i^{*}$ is surjective thus implying, $${\rm Hom}_{{\rm Sp}_{n}(F)}\big(Z(\Delta_{1},\Delta_{2}), \mathbb C\big)=0.$$
Thus we have the lemma if $k\leq n$. Since an irreducible representation has a symplectic period if and only if its contragredient has so, we have the lemma in the case $k>n$.
\end{proof}

\section{Analysis in the ${\rm GL}_{4}(F)$ case: proof of Theorem \ref{Gl4 result}}\label{analysis Gl4}
In this section we prove Theorem \ref{Gl4 result}. We begin with the following lemma.
\begin{lemma}\label{initial reduction lemma Gl4}
Let $\theta$ be an irreducible representation of ${\rm GL}_{4}(F)$ with a symplectic period. Then there exists $\pi_{i}\in Irr ({\rm GL}_{2}(F)),\ i=1,2$ such that $\theta$ appears as a quotient of $\pi_{1}\times\pi_{2}$.
\end{lemma}
\begin{proof}
If $\theta$ is a supercuspidal representation of ${\rm GL}_{4}(F)$, it is generic and hence by Theorem \ref{strong uniqueness} it doesn't have a symplectic period. Thus $\theta$ appears as a quotient of either $\chi_{1}\times\theta_{3}$, $\theta_{3}\times \chi_{1}$ or $\pi_{1}\times\pi_{2}$ (where $\chi_{1}\in Irr(\textrm{GL}_{1}(F))$, $\theta_{3}\in Irr(\textrm{GL}_{3}(F))$ and $\pi_{1},\pi_{2}\in Irr(\textrm{GL}_{2}(F))$). In the last case we have nothing left to prove. If $\theta$ is a quotient of $\theta_{3}\times \chi_{1}$, by Lemma \ref{duality lemma}, $\tilde{\theta}$ is a quotient of $\tilde{\chi_{1}}\times\tilde{\theta_{3}}$. Since an irreducible representation has symplectic period if and only if its contragredient has so, by applying Lemma \ref{duality lemma} again we are reduced to the first case. So assume $\theta$ is a quotient of $\chi_{1}\times\theta_{3}$. Now if $\theta_{3}$ is cuspidal, $\chi_{1}\times\theta_{3}$ is irreducible and generic. Hence by the disjointness of the symplectic and Whittaker models it cannot have a symplectic period. Thus assume $\theta_{3}$ isn't cuspidal.
 
Now, since $\theta_{3}$ isn't cuspidal, it is a quotient of one of the representations of the form $\chi_{1}'\times\delta_{2}$, $\delta_{2}\times\chi_{1}'$ or $\chi_{1}'\times\chi_{1}''\times\chi_{1}'''$ where $\chi_{1}',\chi_{1}'',\chi_{1}'''$ 's are characters of $\textrm{GL}_{1}(F)$ and $\delta_{2}$ is a supercuspidal of $\textrm{GL}_{2}(F)$. 

In the first case, $\chi_{1}\times\theta_{3}$ is a quotient of $\chi_{1}\times\chi_{1}'\times\delta_{2}$. If $\chi_{1}\times\chi_{1}'$ is irreducible, the lemma is proved. If not, $\chi_{1}\times\chi_{1}'\times\delta_{2}$ is glued from $Z(\chi_{1}\times\chi_{1}')\times\delta_{2}$ and $Q(\chi_{1}\times\chi_{1}')\times\delta_{2}$ where $Z(\chi_{1}\times\chi_{1}')$ and $Q(\chi_{1}\times\chi_{1}')$ are respectively the unique irreducible submodule and unique irreducible quotient of $\chi_{1}\times\chi_{1}'$. Thus any irreducible quotient of $\chi_{1}\times\theta_{3}$ has to be a quotient of one of the two. 

In the second case, since $\delta_{2}\times\chi_{1}'$ is irreducible, $\chi_{1}\times\delta_{2}\times\chi_{1}'\cong \chi_{1}\times\chi_{1}'\times\delta_{2}$. Thus we are back to the first case.

In the third case, if both $\chi_{1}\times\chi_{1}'$ and $\chi_{1}''\times\chi_{1}'''$ are irreducible we are done. In case at least one of them is reducible, we get the lemma by breaking $\chi_{1}\times\chi_{1}'\times\chi_{1}''\times\chi_{1}'''$, as in the first case, into subquotients of the required form.

\end{proof}
By this lemma, it is enough to consider representations of the form $\pi_{1}\times\pi_{2}$ where each $\pi_{i},i=1,2$ is an irreducible representation of ${\rm GL}_{2}(F)$. If $\pi=\pi_{1}\times\pi_{2}$ has an $H$-distinguished quotient, then $\pi$ itself is $H$-distinguished. By Mackey theory we get that $(\pi_{1}\times\pi_{2})|_{{\rm Sp}_{2}(F)}$ is glued from the two subquotients, $${\rm ind}^{H}_{H_{2,0}}(\delta_{P_{2,2}}^{1/2}\pi_{1}\otimes\pi_{2}|_{H_{2,0}}) \ \textrm{and} \ {\rm ind}^{H}_{H_{2,2}}(\delta_{P_{2,2}}^{1/2}\pi_{1}\otimes\pi_{2}|_{H_{2,2}}).$$Analyzing the two subquotients (using \ref{general case eqn}), it is easy to see that the necessary conditions for $\pi$ to have a symplectic period are, either $\pi_{1},\pi_{2}$ are characters of ${\rm GL}_{2}(F)$ or if $\pi_{2}\cong \nu^{-1}\pi_{1}$. Any irreducible representation of ${\rm GL}_{2}(F)$ is either a supercuspidal, a character, an irreducible principal series or a twist of the Steinberg representation. Thus any irreducible ${\rm Sp}_{4}(F)$-distinguished representation occurs as a quotient of one of the representations listed in the next proposition.
\begin{proposition}\label{analysis Gl4 proposition}
Let $\theta$ be an irreducible admissible representation of ${\rm GL}_{4}(F)$ with a symplectic period. Then $\theta$ occurs as a quotient of one of the following representations $\pi$ of ${\rm GL}_{4}(F)$:\\
1) $\pi=\chi_{2}\times\chi_{2}^{'}$ where $\chi_{2},\chi_{2}^{'}$ are characters of ${\rm GL}_{2}(F)$.\\
2) $\pi=\sigma_{2} \times\nu^{-1}\sigma_{2}$ where $\sigma_{2}$ is a supercuspidal of ${\rm GL}_{2}(F)$.\\
3) $\pi=\chi_{1}\times\chi_{1}^{'}\times\nu^{-1}\chi_{1}\times\nu^{-1}\chi_{1}^{'}$ where $\chi_{1},\chi_{1}^{'}$ are characters of $F^{\times}$ and $\chi_{1}\times\chi_{1}^{'}$ is an irreducible principal series.\\
4) $\pi=Q([\chi_{1}\nu^{-1/2},\chi_{1}\nu^{1/2}])\times Q([\chi_{1}\nu^{-3/2},\chi_{1}\nu^{-1/2}])$ where $\chi_{1}$ is a character of $F^{\times}$.\\
\end{proposition}

Now we come to the theorem in the ${\rm GL}_{4}(F)$ case. We state and prove an equivalent version of Theorem \ref{Gl4 result} in terms of Zelevinsky classification.
\begin{theorem}\label{Gl4 result1}
Following is the complete list of irreducible admissible representations $\theta$ of ${\rm GL}_{4}(F)$ with a symplectic period:\\
1) $\theta=Z([\sigma_{2},\nu\sigma_{2}])$ where $\sigma_{2}$ is a cuspidal representation of ${\rm GL}_{2}(F)$.\\ 
2) $\theta=Z(\Delta_{1}, \Delta_{2})$ where $\Delta_{1}=[\chi_{1}\nu^{-1/2},\chi_{1}\nu^{1/2}]$ and $\Delta_{2}=[\chi_{1}\nu^{-3/2},\chi_{1}\nu^{-1/2}]$. ($\chi_{1}$ is a character of $F^{\times}$)\\ 
3) $\theta=$ character of ${\rm GL}_{4}(F)$.\\
4) $\theta=\chi_{2}\times \chi_{2}^{'}$ where $\chi_{2},\chi_{2}^{'}$ are characters of ${\rm GL}_{2}(F)$.
\end{theorem}
\begin{proof}
The strategy of the proof is to consider each representation in the list of Proposition \ref{analysis Gl4 proposition} and to check for all irreducible quotients of each one of them, whether or not, they have a symplectic period. 
\paragraph{{\bf Case 1: $\pi=\chi_{2}\times\chi_{2}^{'}$ .}}
If $\chi_{2}\times\chi_{2}^{'}$ is irreducible, $\theta=\pi$ has a symplectic period by Proposition \ref{main prop}. So assume otherwise. Let $\chi_{2}= Z([\chi_{1}\nu^{-1/2},\chi_{1}\nu^{1/2}])$ and $\chi_{2}^{'}= Z([\chi_{1}^{'}\nu^{-1/2},\chi_{1}^{'}\nu^{1/2}])$. This has following four sub-cases.
\paragraph{1) $\chi_{1}=\chi_{1}^{'}\nu$}
In this case, $\pi=Z([\chi_{1}^{'}\nu^{1/2},\chi_{1}^{'}\nu^{3/2}])\times Z([\chi_{1}^{'}\nu^{-1/2},\chi_{1}^{'}\nu^{1/2}])$. By Theorem \ref{product of sub}, it has a unique irreducible quotient which is $\newline$$\theta=Z([\chi_{1}^{'}\nu^{-1/2},\chi_{1}^{'}\nu^{3/2}])\times \chi_{1}^{'}\nu^{1/2}$. By Theorem 3.7 of \cite{Offen-Sayag2}, it has a mixed Klyachko model. Hence by Theorem \ref{strong uniqueness}, this doesn't have a symplectic period. 
\paragraph{2) $\chi_{1}=\chi_{1}^{'}\nu^{-1}$}
In this case, $\pi=Z([\chi_{1}^{'}\nu^{-3/2},\chi_{1}^{'}\nu^{-1/2}])\times Z([\chi_{1}^{'}\nu^{-1/2},\chi_{1}^{'}\nu^{1/2}])$. By Theorem \ref{product of sub}, this has a unique irreducible quotient $\newline$$\theta=\chi_{1}^{'}Z([\nu^{-3/2},\nu^{-1/2}],[\nu^{-1/2},\nu^{1/2}])$ which has a symplectic period by Lemma \ref{existence of symplectic period lemma}. Note that $\theta$ is a twist of $U(St_{2},2)$ and the fact that it has a symplectic period also follows from Proposition \ref{main prop}.
\paragraph{3) $\chi_{1}=\chi_{1}^{'}\nu^{2}$}
In this case, $\pi=Z([\chi_{1}^{'}\nu^{3/2},\chi_{1}^{'}\nu^{5/2}])\times Z([\chi_{1}^{'}\nu^{-1/2},\chi_{1}^{'}\nu^{1/2}])$. This has a unique irreducible quotient $\theta=Z([\chi_{1}^{'}\nu^{-1/2},\chi_{1}^{'}\nu^{5/2}])$. Thus $\theta$ is the character $\chi_{1}\nu$ of ${\rm GL}_{4}(F)$ and has a symplectic period.
\paragraph{4) $\chi_{1}=\chi_{1}^{'}\nu^{-2}$}
In this case, $\pi=Z([\chi_{1}^{'}\nu^{-5/2},\chi_{1}^{'}\nu^{-3/2}])\times Z([\chi_{1}^{'}\nu^{-1/2},\chi_{1}^{'}\nu^{1/2}])$ which by Theorem \ref{product of sub}, has a unique irreducible quotient $\newline$$\theta=Z([\chi_{1}^{'}\nu^{-5/2},\chi_{1}^{'}\nu^{-3/2}],[\chi_{1}^{'}\nu^{-1/2},\chi_{1}^{'}\nu^{1/2}])$. By Lemma \ref{juxt no symplectic period lemma}, it doesn't have a symplectic period.

\paragraph{{\bf Case 2: $\pi=\sigma_{2} \times\nu^{-1}\sigma_{2}$.}}
In this case, $\pi$ has a unique irreducible quotient $U(\nu^{-1/2}\sigma_{2},2)\cong Z([\nu^{-1}\sigma_{2},\sigma_{2}])$. By Proposition \ref{main prop} it has a symplectic period.

\paragraph{{\bf Case 3: $\pi=\chi_{1}\times\chi_{1}^{'}\times\nu^{-1}\chi_{1}\times\nu^{-1}\chi_{1}^{'}$ where $\chi_{1}\times\chi_{1}^{'}$ is irreducible.}}
There are two further sub-cases:
\paragraph{{\bf 1) $\chi_{1}^{'}\times\chi_{1}\nu^{-1}$ is irreducible.}}
This again can be broken down into two sub-cases. 
\paragraph{1a) $\chi_{1}^{'}\neq\chi_{1}\nu^{2}$ }
In this case, $\pi \cong \chi_{1}\times\nu^{-1}\chi_{1}\times\chi_{1}^{'}\times\nu^{-1}\chi_{1}^{'}$. Under the given assumptions the ``does not precede'' condition (see Section \ref{Preliminaries on segments}) is satisfied and so $\pi$ has a unique irreducible quotient. Clearly, $\pi$ has $Z([\nu^{-1}\chi_{1},\chi_{1}]) \times Z([\nu^{-1}\chi_{1}^{'},\chi_{1}^{'}])$ as a quotient. If its irreducible, it has a symplectic period by Proposition \ref{main prop} and has already been accounted for in case 1. So assume the contrary. In that case the segments are linked. But the assumptions that $\chi_{1}^{'}\times\chi_{1}\nu^{-1}$ is irreducible together with $\chi_{1}^{'}\neq\chi_{1}\nu^{2}$ forces a contradiction. Hence irreducibility of $Z([\nu^{-1}\chi_{1},\chi_{1}]) \times Z([\nu^{-1}\chi_{1}^{'},\chi_{1}^{'}])$ is the only possibility.
\paragraph{1b) $\chi_{1} ^{'}= \chi_{1}\nu^{2}$ } 
In this case, $\pi\cong\chi_{1}^{'}\nu^{-2}\times\chi_{1}^{'}\times\chi_{1}^{'}\nu^{-3}\times\chi_{1}^{'}\nu^{-1}\cong \chi_{1}^{'}\nu^{-2}\times\chi_{1}^{'}\nu^{-3}\times\chi_{1}^{'}\times\chi_{1}^{'}\nu^{-1}$. Note that this representation has $\tau=Z([\chi_{1}^{'}\nu^{-3},\chi_{1}^{'}\nu^{-2}])\times Z([\chi_{1}^{'}\nu^{-1},\chi_{1}^{'}])$ as a quotient. Since the cuspidal support of $\pi$ is multiplicity free it has a unique irreducible quotient (by Proposition 2.10 of \cite{Z}) and so any $H$-distinguished irreducible quotient of $\pi$ is a quotient of $\tau$. Thus they have already been accounted for in case 1 of this proof. 

\paragraph{{\bf 2) $\chi_{1}^{'} \times \chi_{1}\nu^{-1}$ is reducible.}}
The condition is satisfied if and only if either $\chi_{1}=\chi_{1}^{'}$ or $\chi_{1}= \chi_{1}^{'}\nu^{2}$. Again, we will deal with the cases separately. 
\paragraph{2a) $\chi_{1} = \chi_{1}^{'}$ }
The representation is of the form $\chi_{1}\times\chi_{1}\times\chi_{1}\nu^{-1}\times\chi_{1}\nu^{-1}$. Since it satisfies the ``does not precede'' condition (see Section \ref{Preliminaries on segments}) it has a unique irreducible quotient. It can be easily seen that $\theta=Z([\chi_{1}\nu^{-1},\chi_{1}]) \times Z([\chi_{1}\nu^{-1},\chi_{1}])$ is an irreducible quotient of this representation (and so is the unique one). $\theta$ has symplectic period and has already been accounted for in case 1 of this proof.
\paragraph{2b) $\chi_{1} = \chi_{1}^{'}\nu^{2} $}
In this case, $\pi\cong\chi_{1}^{'}\nu^{2}\times\chi_{1}^{'}\times\chi_{1}^{'}\nu\times\chi_{1}^{'}\nu^{-1}\cong \chi_{1}^{'}\times\chi_{1}^{'}\nu^{-1}\times\chi_{1}^{'}\nu^{2}\times\chi_{1}^{'}\nu$. By an argument similar to the one used in case 3.1.b we conclude that this representation has already been accounted for in case 1 of this proof.

\paragraph{{\bf Case 4: $\pi=Q([\chi_{1}\nu^{-1/2},\chi_{1}\nu^{1/2}])\times Q([\chi_{1}\nu^{-3/2},\chi_{1}\nu^{-1/2}])$}}
By Theorem \ref{product of quotient}, $\pi$ has a unique irreducible quotient $\chi_{1}Q([\nu^{-3/2},\nu^{-1/2}],[\nu^{-1/2},\nu^{1/2}])$. As mentioned earlier in case 1.4, it is a twist of $U(St_{2},2)$ and has a symplectic period (by Proposition \ref{main prop}).

\end{proof}

\section{Analysis in the ${\rm GL}_{6}(F)$ case}\label{Analysis of}
In this section we obtain the theorem for ${\rm GL}_{6}(F)$. The following lemma reduces the analysis to representations of the form $\pi_{1}\times\pi_{2}\times\pi_{3}$ where the $\pi_{i}$'s are irreducible representations of ${\rm GL}_{2}(F)$.
\begin{lemma}\label{initial reduction lemma}
Let $\theta$ be an irreducible representation of ${\rm GL}_{6}(F)$ with a symplectic period. Then either $\theta$ is of the form $Z([\sigma_{3},\nu\sigma_{3}])$, where $\sigma_{3}$ is a supercuspidal representation of ${\rm GL}_{3}(F)$ or it occurs as a subquotient of a representation of the form $\pi_{1}\times \pi_{2}\times \pi_{3}$ where $\pi_{i}\in Irr({\rm GL}_{2}(F))$ for $i=1,2,3$.
\end{lemma}
\begin{proof}
Since supercuspidal representations are generic, they don't have a symplectic period. Thus $\theta$ appears as a subquotient of $\tau_{1}\times\tau_{2}$ where $\tau_{1}$ and $\tau_{2}$ are irreducible representations of $\textrm{GL}_{k}(F)$ and $\textrm{GL}_{6-k}(F)$ respectively (for $k\leq 3$).
\paragraph{{\bf Case 1: $k=1$ }}
If $\tau_{2}$ is a cuspidal representation of ${\rm GL}_{5}(F)$, since $\tau_{1}$ is a character, $\tau_{1}\times\tau_{2}$ is irreducible and generic. Thus $\tau_{2}$ occurs as a subquotient of a representation induced from a maximal parabolic of ${\rm GL}_{5}(F)$. So $\theta$ is either a subquotient of $\tau_{1}\times\chi\times\tau\ (\chi\in Irr({\rm GL}_{1}(F)),\ \tau\in Irr({\rm GL}_{4}(F)))$ or $\tau_{1}\times\tau^{'}\times\tau^{''}\ (\tau^{'}\in Irr({\rm GL}_{2}(F)),\ \tau^{''}\in Irr({\rm GL}_{3}(F)))$. Thus $\theta$ is either a subquotient of $\theta_{1}\times\tau$ (where $\theta_{1}\in \ Irr({\rm GL}_{2}(F))$) or of $\theta_{2}\times\tau^{''}$ (where $\theta_{2}\in \ Irr({\rm GL}_{3}(F))$) thus reducing the lemma to the next two cases.
\paragraph{{\bf Case 2: $k=2$ }}
If $\tau_{2}$ is a cuspidal representation of ${\rm GL}_{4}(F)$, $\tau_{1}\times\tau_{2}$ is irreducible and doesn't have a symplectic period by Lemma \ref{irr}. Thus as earlier, $\tau_{2}$ occurs as a subquotient of a representation induced from a maximal parabolic of ${\rm GL}_{4}(F)$. So $\theta$ is either a subquotient of $\tau_{1}\times\chi\times\tau\ (\chi\in Irr({\rm GL}_{1}(F)),\ \tau\in Irr({\rm GL}_{3}(F)))$ or $\tau_{1}\times\tau^{'}\times\tau^{''}\ (\tau^{'}\in Irr({\rm GL}_{2}(F)),\ \tau^{''}\in Irr({\rm GL}_{2}(F)))$. In the first scenario $\theta$ occurs as a subquotient of $\theta_{1}\times\theta_{2}$ (where $\theta_{1},\ \theta_{2}\in Irr({\rm GL}_{3}(F))$) (thus reducing the lemma to case 3) while in the second we have the lemma.
\paragraph{{\bf Case 3: $k=3$ }}
We will first show that in case either of $\tau_{1},\tau_{2}$ (say $\tau_{1}$) is cuspidal then $\theta$ is of the form $Z([\sigma_{3},\nu\sigma_{3}])$. Choose $\tau_{2}^{'} \in Irr({\rm GL}_{3}(F))$ such that $\theta$ is a quotient of either $\tau_{1}\times\tau_{2}^{'}$ or $\tau_{2}^{'}\times\tau_{1}$. Assume the former. Then $\tau_{1}\times\tau_{2}^{'}$ also has a non-trivial ${\rm Sp}_{3}(F)$-invariant linear form. Now, $\tau_{1}\times\tau_{2}^{'}|_{{\rm Sp}_{3}(F)}$ is glued from 
$${\rm ind}^{H}_{H_{3,3}}(\delta_{P_{3,3}}^{1/2}\tau_{1}\otimes\tau_{2}'|_{H_{3,3}}) \ \textrm{and} \ {\rm ind}^{H}_{H_{3,1}}(\delta_{P_{3,3}}^{1/2}\tau_{1}\otimes\tau_{2}'|_{H_{3,1}}).$$
Since $\tau_{1}$ is cuspidal, by Eqn. \ref{general case eqn},
$${\rm Hom}_{H}({\rm ind}^{H}_{H_{3,1}}(\delta_{P_{3,3}}^{1/2}\tau_{1}\otimes\tau_{2}'|_{H_{3,1}}),\mathbb C)=0.$$
So, $${\rm Hom}_{H}({\rm ind}^{H}_{H_{3,3}}(\delta_{P_{3,3}}^{1/2}\tau_{1}\otimes\tau_{2}'|_{H_{3,3}}),\mathbb C)\neq 0$$ which is true iff $\tau_{2}^{'}=\nu^{-1}\tau_{1}$, again by Eqn. \ref{general case eqn} and a theorem of Gelfand and Kazhdan (Theorem 7.3 of \cite{BZ1}). Thus $\theta=Z([\nu^{-1}\tau_{1},\tau_{1}])$. In case, $\theta$ is a quotient of $\tau_{2}^{'}\times\tau_{1}$ replacing $\theta$ by $\tilde{\theta}$ gives us the desired result.

Thus assume now that none of the two are cuspidal. Then $\exists \ \chi_{i}$, $\theta^{'}_{i}$ ($i=1,2$) such that $\tau_{i}$ is a subquotient of $\chi_{i}\times\theta^{'}_{i}$ (where $\chi_{i}\in Irr({\rm GL}_{1}(F)),\ \theta_{i}^{'}\in Irr({\rm GL}_{2}(F))$). Thus $\theta$ is a subquotient of $\chi_{1}\times\chi_{2}\times\theta^{'}_{1}\times\theta^{'}_{2}$ and hence the lemma is proved.  
\end{proof}
Next we prove a hereditary property for ${\rm GL}_{6}(F)$ using the classification theorem for ${\rm GL}_{4}(F)$.  
\begin{proposition}\label{hereditary for Gl6}
Let $\pi_{1} \in Irr({\rm GL}_{2}(F))$ and $\pi_{2} \in Irr({\rm GL}_{4}(F))$ be two irreducible representations with symplectic periods. Then $\pi_{1}\times\pi_{2}$ has a symplectic period. Similarly, if $\pi_{i},\ i=1,2,3$ are irreducible representations of ${\rm GL}_{2}(F)$ with a symplectic period, then $\pi_{1}\times\pi_{2}\times\pi_{3}$ has a symplectic period.
\end{proposition}
\begin{proof}
Any irreducible representation $\pi$ of ${\rm GL}_{2}(F)$ having a symplectic period is a character while by Theorem \ref{Gl4 result} any such representation of ${\rm GL}_{4}(F)$ is either a character, irreducible product of two characters of ${\rm GL}_{2}(F)$ or a representation of the form $U(\delta,2)$. The proposition now follows from Proposition \ref{main prop}.  
\end{proof}

The following lemma is a consequence of Lemma \ref{irr} and the fact that cuspidal representations are generic (and hence not symplectic).
\begin{lemma}\label{no cuspidal lemma}
Let $\pi_{i},(1\leq i \leq 3)$ be irreducible admissible representations of ${\rm GL}_{2}(F)$. If one or more of the $\pi_{i}$'s are cuspidal and $\theta$ is an ${\rm Sp}_{3}(F)$-distinguished subquotient of $\pi=\pi_{1}\times\pi_{2}\times\pi_{3}$ then it is of the form $\chi_{2}\times Z([\sigma_{2},\nu\sigma_{2}])$ where $\chi_{2}$ and $\sigma_{2}$ are a character and a supercuspidal of ${\rm GL}_{2}(F)$ respectively.
\end{lemma}
\begin{proof}
Without loss of generality let $\pi_{3}$ be a supercuspidal. Call it $\sigma_{2}$. Now there can be three cases depending on $\pi_{1}$ and $\pi_{2}$.
\paragraph{{\bf Case 1: None of $\pi_{1}$ and $\pi_{2}$ are cuspidal.}}
In this case $\sigma_{2}$ is not in the cuspidal support of $\pi_{1}\times\pi_{2}$ and hence any irreducible subquotient of $\pi$ is of the form $\sigma_{2}\times J$ where $J$ is an irreducible subquotient of $\pi_{1}\times\pi_{2}$. By Lemma \ref{irr}, it doesn't have a symplectic period.
\paragraph{{\bf Case 2: Both $\pi_{1}$ and $\pi_{2}$ are cuspidal.}}
In this case $\pi$ is of the form $\sigma_{2}\times\sigma_{2}^{'}\times\sigma_{2}^{''}$. In case none of the pairs are linked or there is exactly one linked pair among the 3, then again by Lemma \ref{irr}, $\pi$ doesn't have an ${\rm Sp}_{3}(F)$-distinguished irreducible subquotient. So $\pi$ has to be either of the form $\sigma_{2}\times\nu\sigma_{2}\times\nu\sigma_{2}$, $\sigma_{2}\times\sigma_{2}\times\nu\sigma_{2}$ or $\sigma_{2}\times\nu\sigma_{2}\times\nu^{2}\sigma_{2}$ (up to a permutation of $\pi_{i}$'s). 

If $\pi=\sigma_{2}\times\nu\sigma_{2}\times\nu^{2}\sigma_{2}$ (or a permutation), then it has 4 irreducible subquotients. Of these, $Q([\sigma_{2},\nu^{2}\sigma_{2}])$ is generic and $Z([\sigma_{2},\nu^{2}\sigma_{2}])$ doesn't have a symplectic period (by Theorem \ref{unitary symplectic}). Now consider the subquotient $Z([\sigma_{2}],[\nu\sigma_{2},\nu^{2}\sigma_{2}])$. It is the unique irreducible quotient of the representation $\mu_{1}\times\mu_{2}$ where $\mu_{1}=\sigma_{2}$ and $\mu_{2}=Z([\nu\sigma_{2},\nu^{2}\sigma_{2}])$. Now, using Eqn. \ref{general case eqn}, it can be easily checked that ${\rm Hom}_{H}\big({\rm ind}^{H}_{H_{2,0}}(\delta_{P_{2,4}}^{1/2}\mu_{1}\otimes\mu_{2}|_{H_{2,0}}),\mathbb C\big)$ and  ${\rm Hom}_{H}\big({\rm ind}^{H}_{H_{2,2}}(\delta_{P_{2,4}}^{1/2}\mu_{1}\otimes\mu_{2}|_{H_{2,2}}),\mathbb C\big)$ are both 0, thus implying ${\rm Hom}_{H}\big(\mu_{1}\times\mu_{2},\mathbb C\big)=0$. So, $Z([\sigma_{2}],[\nu\sigma_{2},\nu^{2}\sigma_{2}])$ doesn't have a symplectic period and by taking contragredients we conclude that neither does $Z([\nu^{2}\sigma_{2}],[\sigma_{2},\nu\sigma_{2}])$. Thus $\pi$ doesn't have any irreducible subquotient carrying a symplectic period.

If $\pi=\sigma_{2}\times\nu\sigma_{2}\times\nu\sigma_{2}$ (or a permutation), it is glued from the irreducible representations $\nu\sigma_{2}\times Z([\sigma_{2},\nu\sigma_{2}])$ and $\nu\sigma_{2}\times Q([\sigma_{2},\nu\sigma_{2}])$. As in the above paragraph, taking $\mu_{1}=\nu\sigma_{2}$ and $\mu_{2}=Z([\sigma_{2},\nu\sigma_{2}])$ and using Eqn. \ref{general case eqn}, it can be easily checked that $\nu\sigma_{2}\times Z([\sigma_{2},\nu\sigma_{2}])$ doesn't have a symplectic period. The representation $\nu\sigma_{2}\times Q([\sigma_{2},\nu\sigma_{2}])$ is generic and hence doesn't have a symplectic period by (by Theorem \ref{strong uniqueness}). Similarly $\sigma_{2}\times\sigma_{2}\times\nu\sigma_{2}$ (or any of its permutations) cannot have an ${\rm Sp}_{3}(F)$-distinguished subquotient either.

\paragraph{{\bf Case 3: Exactly one of $\pi_{1}$ and $\pi_{2}$ is cuspidal.}}
Up to a permutation, $\pi$ then is a representation of the form $\sigma_{2}\times \sigma_{2}^{'}\times\theta^{'}$ where $\theta^{'}$ is an irreducible representation of ${\rm GL}_{2}(F)$ which isn't a supercuspidal. If $\sigma_{2}^{'}$ and $\sigma_{2}$ are linked and $\theta^{'}$ is a character then $\pi$ has an ${\rm Sp}_{3}(F)$-distinguished subquotient of the required form. Otherwise again by Lemma \ref{irr}, it doesn't have one. 
\end{proof}

Thus it reduces the analysis to the cases where each $\pi_{i}$ is either a character, an irreducible principal series or a twist of the Steinberg. Note that (up to a permutation of $\pi_{i}$'s) there are 10 possible cases. Next we show that if at least two of the $\pi_{i}$'s are irreducible principal series representations, then we need not consider those cases. This reduces the analysis to the rest 7 cases. 
\begin{lemma}\label{no two irrd ps}
Let $\pi_{i},(1\leq i \leq 3)$ be irreducible admissible representations of ${\rm GL}_{2}(F)$ such that none of them are cuspidal. If two or more of the $\pi_{i}$'s are irreducible principal series representations and $\theta$ is an ${\rm Sp}_{3}(F)$-distinguished subquotient of $\pi=\pi_{1}\times\pi_{2}\times\pi_{3}$ then it also appears as a subquotient of $\pi=\pi_{1}\times\pi_{2}\times \pi_{3}$ where at most one of the $\pi_{i}$'s is a principal series representation.
\end{lemma}
\begin{proof}
If $\theta$ is as above, it is a subquotient of a representation of the form $\tau=\chi_{1}\times...\times \chi_{6}$, where each $\chi_{i}$ is a character of ${\rm GL}_{1}(F)$. It is easy to see that Lemma \ref{irr} implies that unless all the $\chi_{i}$'s are in the same cuspidal line, $\theta$ is an irreducible product of a character of ${\rm GL}_{2}(F)$ and an irreducible ${\rm Sp}_{2}(F)$-distinguished representation. We count them in the case when all the three $\pi_{i}$'s are characters. So without loss of generality we can assume the $\chi_{i}$'s to be integral powers of the character $\nu$ of ${\rm GL}_{1}(F)$. Say a character is linked to another character if they are linked as one element segments (see Section \ref{Preliminaries on segments}), i.e. $\nu^{a}$ and $\nu^{b}$ are linked iff $a-b=\pm 1$. If no two of the characters appearing in $\tau$ are linked, it is irreducible and generic and so $\theta$ cannot be its subquotient. So we can assume $\tau\cong 1\times\nu\times\nu^{a}\times \nu^{b}\times\nu^{c}\times\nu^{d}$. 

Now, assume that there is a character among $\nu^{a},...,\nu^{d}$ (say $\nu^{a}$) such that it is not linked to any of the other characters. Collecting all the $\nu^{a}$s together, we see that $\theta=\nu^{a}\times...\times \nu^{a}\times J$ for some irreducible representation $J$ such that $\nu^{a}\times...\times\nu^{a}$ and $J$ satisfy the hypothesis of Lemma \ref{irr}. So $\tau$ cannot have an ${\rm Sp}_{3}(F)$-distinguished subquotient. Thus we further assume that all the characters among $\nu^{a},...,\nu^{d}$ are linked to some other character. 

Note that if there exists a partition of the characters of $\tau$ such that at least two different blocks of the partition consist of linked pairs, $\tau$ is glued from subquotients of the form $\tau_{1}\times\tau_{2}\times \nu^{n_{1}}\times\nu^{n_{2}}$ where $\tau_{i}$ is either a character or a twist of the Steinberg. Thus $\theta$ can also be obtained in the cases when two of the $\pi_{i}$'s are characters, two of them are twists of the Steinberg or one of the $\pi_{i}$'s is a character and another one is a twist of the Steinberg. 

Thus if we show that, under the hypothesis that any two of the characters of $\tau$ are linked and such a partition of them doesn't exist, $\tau$ cannot have an irreducible $H$-distinguished subquotient we are done. Lemma \ref{no disjoint linked} precisely does that.
\end{proof}  
\begin{lemma}\label{no disjoint linked}
Call the characters $\nu^{a}$ and $\nu^{b}$ to be linked iff $a-b=\pm 1$. Let $\tau\cong 1\times\nu\times\nu^{a}\times \nu^{b}\times\nu^{c}\times\nu^{d}$ be such that every character of it is linked to some other character. Furthermore assume that there doesn't exist any partition of the characters such that at least two different blocks of it consist of linked pairs. Then $\tau$ cannot have an irreducible ${\rm Sp}_{3}(F)$-distinguished subquotient.
\end{lemma}
\begin{proof}
Let if possible $\theta$ be an ${\rm Sp}_{3}(F)$-distinguished subquotient of $\tau$. The hypothesis of the lemma implies that the cuspidal support of $\tau$ can have at most $\nu^{-1}$ or $\nu^{2}$ along with 1 and $\nu$. Moreover, $\nu^{-1}$ and $\nu^{2}$ can both not be there simultaneously and in case the support only consists of 1 and $\nu$, $\tau$ is one of the representations $1\times 1\times 1\times 1\times 1\times \nu$ or $1\times \nu\times \nu\times \nu\times \nu\times \nu$ (up to a permutation of the characters). 

In case, $\tau$ has $\nu^{-1}$ in the cuspidal support, 1 can be there only with multiplicity one and so the only possible forms for $\tau$, up to a permutation of the characters, are $\nu^{-1}\times 1\times \nu\times \nu\times \nu\times \nu$, $\nu^{-1}\times \nu^{-1}\times 1\times\nu\times \nu\times \nu$, $\nu^{-1}\times \nu^{-1}\times \nu^{-1}\times 1\times\nu\times \nu$ or $\nu^{-1}\times \nu^{-1}\times \nu^{-1}\times \nu^{-1}\times 1\times\nu$. Consider first the representation $\nu^{-1}\times \nu^{-1}\times \nu^{-1}\times \nu^{-1}\times 1\times\nu$. There exists a permutation of $\nu^{-1},...,\nu^{-1},1,\nu$ such that $\theta$ is a quotient of the representation that is obtained by taking the product in that order. It is an easy calculation (using similar arguments as in case 1(a) to follow i.e. in the case when all three $\pi_{i}$'s are characters) to check that no permutation gives a product which is $H$-distinguished. Thus $\theta$ cannot have a symplectic period which is a contradiction. So $\tau$ cannot be $\nu^{-1}\times \nu^{-1}\times \nu^{-1}\times \nu^{-1}\times 1\times\nu$ (or any permutation of the characters). Similarly one checks that $\tau$ cannot be $\nu^{-1}\times \nu^{-1}\times \nu^{-1}\times 1\times\nu\times \nu$ (or any permutation of the characters). Since the other two representations are contragredients of the above two representations, they cannot have any $H$-distinguished subquotients either. Thus $\tau$ cannot be any permutation of one of them either and we conclude that $\tau$ cannot have $\nu^{-1}$ in its cuspidal support.

Observe that the possible values of $\tau$ if its cuspidal support has $\nu^{2}$ instead of $\nu^{-1}$ can all be obtained by appropriately twisting the contragredients of the ones obtained in the $\nu^{-1}$ case. So $\tau$ cannot be one of them either and hence cannot have $\nu^{2}$ in its cuspidal support.
 
Thus $\tau$ can only have $1$'s and $\nu$'s in its cuspidal support. If $\tau=1\times 1\times 1\times 1\times 1\times \nu$ (or any permutation of the characters), it is glued from $Z([1,\nu])\times 1\times 1\times 1\times 1$ and $Q([1,\nu])\times 1\times 1\times 1\times 1$. For the first one take $\mu_{1}=Z([1,\nu]),\mu_{2}=1\times 1\times 1\times 1$ and use Eqn. \ref{general case eqn} (as in case 1(a) to follow) to conclude that it doesn't have a symplectic period. The second one is generic and hence also cannot have symplectic period. Thus again $\tau$ cannot have $\theta$ as a subquotient and so it cannot be a permutation of $1\times 1\times 1\times 1\times 1\times \nu$. Taking contragredients  we conclude that it cannot be a permutation of $1\times \nu\times \nu\times \nu\times \nu\times \nu$ either. This shows that even 1 and $\nu$ cannot be in the cuspidal support of $\tau$. This is a contradiction to our initial assumption that $\tau$ has an ${\rm Sp}_{3}(F)$-distinguished subquotient.
\end{proof}
Thus we need to analyze only the remaining 7 cases. We begin with the following: 
\paragraph{{\bf Case 1: All the three $\pi_{i}$'s are characters.}}$\newline$
The representations in this case are of the form $$\pi=Z([\chi_{1},\chi_{1}\nu])\times Z([\chi_{1}^{'},\chi_{1}^{'}\nu])\times Z([\chi_{1}^{''},\chi_{1}^{''}\nu]).$$

If there are no links among the three segments, then the representation is an irreducible product of characters of ${\rm GL}_{2}(F)$ and the representation is symplectic by Proposition \ref{main prop}. Assume now that there is exactly one link. Without loss of generality we can assume that $[\chi_{1}^{'},\chi_{1}^{'}\nu]$ and $[\chi_{1}^{''},\chi_{1}^{''}\nu]$ are linked, and that $[\chi_{1},\chi_{1}\nu]$ is not linked to either. Clearly then, $\chi_{1}\neq\chi_{1}^{'},\chi_{1}^{''}$. So, $[\chi_{1},\chi_{1}\nu]$ is disjoint and not linked to either $[\chi_{1}^{'},\chi_{1}^{'}\nu]$ or $[\chi_{1}^{''},\chi_{1}^{''}\nu]$. Observe that if a segment $\Delta_{1}$ is not linked to $\Delta_{2}$ and $\Delta_{3}$ (where $\Delta_{2}$ and $\Delta_{3}$ are linked), it is not linked to $\Delta_{2}\cup\Delta_{3}$ or $\Delta_{2}\cap\Delta_{3}$ either. So by Theorem 7.1 and Proposition 8.5 of \cite{Z}, each irreducible subquotient of $\pi$ is of the form $Z([\chi_{1},\chi_{1}\nu])\times \theta^{'}$, where $\theta^{'}$ is an irreducible subquotient of $Z([\chi_{1}^{'},\chi_{1}^{'}\nu])\times Z([\chi_{1}^{''},\chi_{1}^{''}\nu])$. Moreover if the subquotient is $H$-distinguished, observe then that $Z([\chi_{1},\chi_{1}\nu])$ and $\theta^{'}$ satisfy the hypothesis of Lemma \ref{irr}. Thus by Lemma \ref{irr} any irreducible ${\rm Sp}_{3}(F)$-distinguished subquotient is an irreducible product of $H$-distinguished representations of ${\rm GL}_{2}(F)$ and ${\rm GL}_{4}(F)$. 

Hence we look at the cases where there are at least two links among the segments. Without loss of generality we can assume $\chi_{1}$ to be trivial. Following are then the 8 possible cases:
$\newline a) \ Z([1,\nu])\times Z([\nu,\nu^{2}])\times Z([\nu^{3},\nu^{4}])$
$\newline b) \ Z([1,\nu])\times Z([\nu,\nu^{2}])\times Z([\nu^{2},\nu^{3}])$
$\newline c) \ Z([1,\nu])\times Z([\nu^{2},\nu^{3}])\times Z([\nu^{4},\nu^{5}])$
$\newline d) \ Z([1,\nu])\times Z([\nu,\nu^{2}])\times Z([\nu,\nu^{2}])$
$\newline e) \ Z([1,\nu])\times Z([\nu^{2},\nu^{3}])\times Z([\nu^{2},\nu^{3}])$
$\newline f) \ Z([1,\nu])\times Z([\nu^{2},\nu^{3}])\times Z([\nu^{3},\nu^{4}])$
$\newline g) \ Z([1,\nu])\times Z([1,\nu])\times Z([\nu,\nu^{2}])$
$\newline h) \ Z([1,\nu])\times Z([1,\nu])\times Z([\nu^{2},\nu^{3}])$
$\newline$
In each one of the cases listed above, we evaluate all possible irreducible subquotients (using Theorem 7.1 of \cite{Z}) and individually determine whether they have a symplectic period or not.
\paragraph{{\bf Case a) $\pi=Z([1,\nu])\times Z([\nu,\nu^{2}])\times Z([\nu^{3},\nu^{4}])$ }}
$\newline$In this case, all irreducible subquotients of $\pi$ are $Z([1,\nu],[\nu,\nu^{2}],[\nu^{3},\nu^{4}])$, $Z([1,\nu^{2}],[\nu],[\nu^{3},\nu^{4}])$, $Z([1,\nu^{4}],[\nu])$ and $Z([1,\nu],[\nu,\nu^{4}])$. We now analyze each of these representations one by one.

$\theta=Z([1,\nu],[\nu,\nu^{2}],[\nu^{3},\nu^{4}])$ is the unique irreducible submodule of $Z([\nu^{3},\nu^{4}])\times Z([\nu,\nu^{2}])\times Z([1,\nu])$. Using Lemma \ref{duality lemma} and then taking the contragredients we get that $\theta$ is the unique irreducible quotient of $\pi=Z([1,\nu])\times Z([\nu,\nu^{2}])\times Z([\nu^{3},\nu^{4}])$. Since $Z([\nu,\nu^{2}],[\nu^{3},\nu^{4}])$ is a quotient of $Z([\nu,\nu^{2}])\times Z([\nu^{3},\nu^{4}])$, $\theta$ is also the unique irreducible quotient of $Z([1,\nu])\times Z([\nu,\nu^{2}],[\nu^{3},\nu^{4}])$. 

Let $\mu_{1}=Z([1,\nu])$ and $\mu_{2}=Z([\nu,\nu^{2}],[\nu^{3},\nu^{4}])$.
Now, $\mu_{1}\times\mu_{2}$ is glued from $${\rm ind}^{H}_{H_{2,0}}(\delta^{1/2}_{P_{2,4}}\mu_{1}\otimes\mu_{2}|_{H_{2,0}})\ \textrm{and} \ {\rm ind}^{H}_{H_{2,2}}(\delta^{1/2}_{P_{2,4}}\mu_{1}\otimes\mu_{2}|_{H_{2,2}})$$(see Section \ref{Orbit structures and mackey theory}). Since $\mu_{2}$ doesn't have a symplectic period (by Lemma \ref{juxt no symplectic period lemma}), $${\rm Hom}_{H}\big({\rm ind}^{H}_{H_{2,0}}(\delta^{1/2}_{P_{2,4}}\mu_{1}\otimes\mu_{2}|_{H_{2,0}}),\mathbb C\big)={\rm Hom}_{{\rm Sp}_{1}}\big(\mu_{1},\mathbb C \big)\otimes {\rm Hom}_{{\rm Sp}_{2}}\big(\mu_{2},\mathbb C\big)=0$$(by \ref{general case eqn}). On the other hand, 
$${\rm Hom}_{H}\big({\rm ind}^{H}_{H_{2,2}}(\delta^{1/2}_{P_{2,4}}\mu_{1}\otimes\mu_{2}|_{H_{2,2}}),\mathbb C\big) ={\rm Hom}_{{\rm GL}_{2}\times {\rm Sp}_{1}}\big(\nu^{-1}\mu_{1}\otimes r_{(2,2),(4)}(\mu_{2}),\mathbb C\big).$$
It can be seen by the geometrical lemma, cf. Lemma 2.12 of \cite{BZ2}, that $r_{(2,2),(4)}(Z([\nu,\nu^{2}])\times Z([\nu^{3},\nu^{4}]))$ is glued from the irreducible representations $Z([\nu,\nu^{2}])\otimes Z([\nu^{3},\nu^{4}])$, $Z([\nu^{3},\nu^{4}])\otimes Z([\nu,\nu^{2}])$ and $(\nu\times\nu^{3})\otimes(\nu^{2}\times\nu^{4})$. Jacquet functor being an exact functor, $r_{(2,2),(4)}(\mu_{2})$ is glued from one or more of these terms. It can be checked that substituting $r_{(2,2),(4)}(\mu_{2})$ by each one of these three representations makes the right hand side of the above equation 0.

Thus, we get that $Z([1,\nu])\times Z([\nu,\nu^{2}],[\nu^{3},\nu^{4}])$ doesn't have a symplectic period. If $\theta$ had a symplectic period, this would have given a non-trivial ${\rm Sp}_{3}(F)$-invariant linear functional of $Z([1,\nu])\times Z([\nu,\nu^{2}],[\nu^{3},\nu^{4}])$ (by composing the one for $\theta$ with the quotient map), a contradiction. Hence $\theta$ doesn't have a symplectic period.  

If $\theta=Z([1,\nu^{2}],[\nu],[\nu^{3},\nu^{4}])$, it is the unique irreducible submodule of $Z([\nu^{3},\nu^{4}])\times \nu \times Z([1,\nu^{2}])$. Using Lemma \ref{duality lemma} and taking contragredients, $\theta$ is the unique irreducible quotient of $ Z([1,\nu^{2}]) \times \nu \times Z([\nu^{3},\nu^{4}])$. Applying Lemma \ref{duality lemma} again, $\tilde{\theta}$ is the unique irreducible quotient of $Z([\nu^{-4},\nu^{-3}])\times \nu^{-1}\times Z([\nu^{-2},1])$. By taking $\mu_{1}=Z([\nu^{-4},\nu^{-3}])$ and $\mu_{2}=\nu^{-1}\times Z([\nu^{-2},1])$ and doing a similar calculation as above, we get that $Z([\nu^{-4},\nu^{-3}])\times \nu^{-1}\times Z([\nu^{-2},1])$, and hence $\tilde{\theta}$, doesn't have a symplectic period. Thus even $\theta$ is not ${\rm Sp}_{3}(F)$-distinguished. 

If $\theta=Z([1,\nu^{4}],[\nu])\cong Z([1,\nu^{4}])\times\nu$ (by Proposition 8.5 of \cite {Z}), by Theorem \ref{product of sub}, it is an irreducible quotient of $Z([\nu^{3},\nu^{4}])\times Z([1,\nu^{2}])\times \nu$. Now doing a similar calculation as in the first case by taking $\mu_{1}=Z([\nu^{3},\nu^{4}])$ and $\mu_{2}=Z([1,\nu^{2}])\times \nu$ we get that $Z([\nu^{3},\nu^{4}])\times Z([1,\nu^{2}])\times \nu$, and hence $\theta$, doesn't have a symplectic period.

If $\theta=Z([1,\nu],[\nu,\nu^{4}])$, it has a symplectic period by Lemma \ref{existence of symplectic period lemma}.

Thus we are done with case (a).

\paragraph{{\bf Case b) $\pi=Z([1,\nu])\times Z([\nu,\nu^{2}])\times Z([\nu^{2},\nu^{3}])$ }}
$\newline$In this case, all irreducible subquotients of $\pi$ are $Z([1,\nu],[\nu,\nu^{2}],[\nu^{2},\nu^{3}])$, $Z([1,\nu^{2}],[\nu],[\nu^{2},\nu^{3}])$, $\ Z([1,\nu],[\nu,\nu^{3}],[\nu^{2}])$, $Z([1,\nu^{3}],[\nu],[\nu^{2}])$, $\newline$$Z([1,\nu^{2}],[\nu,\nu^{3}])$ and $Z([1,\nu^{3}],[\nu,\nu^{2}])$. We now analyze each of these representations one by one.

If $\theta=Z([1,\nu],[\nu,\nu^{2}],[\nu^{2},\nu^{3}])$, by Theorem A. 10(iii) of \cite{Tadic1}, $\newline$$\theta=Q([1,\nu^{2}],[\nu,\nu^{3}])$. Twisting $\theta$ by an appropriate power of $\nu$ makes it a unitary representation, which then turns out to be ${\rm Sp}_{3}(F)$-distinguished by Theorem \ref{unitary symplectic}. 

If $\theta=Z([1,\nu^{2}],[\nu],[\nu^{2},\nu^{3}])$, it is the unique irreducible submodule of $Z([\nu^{2},\nu^{3}])\times\nu\times Z([1,\nu^{2}])$. Using Lemma \ref{duality lemma} and then taking the contragredients we get that $\theta$ is the unique irreducible quotient of $\pi=Z([1,\nu^{2}])\times \nu\times Z([\nu^{2},\nu^{3}])$. Using Lemma \ref{duality lemma} again, we get that $\tilde{\theta}$ is the unique irreducible quotient of $Z([\nu^{-3},\nu^{-2}])\times \nu^{-1}\times Z([\nu^{-2},1])$. By taking $\mu_{1}=Z([\nu^{-3},\nu^{-2}])$ and $\mu_{2}=\nu^{-1}\times Z([\nu^{-2},1])$, and doing a similar calculation as in case (a), we get that $Z([\nu^{-3},\nu^{-2}])\times \nu^{-1}\times Z([\nu^{-2},1])$ and hence $\tilde{\theta}$, doesn't have a symplectic period. Thus even $\theta$ is not ${\rm Sp}_{3}(F)$-distinguished.

If $\theta=Z([1,\nu],[\nu,\nu^{3}],[\nu^{2}])$, it can be obtained by twisting the contragredient of $Z([1,\nu^{2}],[\nu],[\nu^{2},\nu^{3}])$ which, as showed in the last paragraph, doesn't have a symplectic period. 

If $\theta=Z([1,\nu^{3}],[\nu],[\nu^{2}])$, it is the unique irreducible submodule of $\nu^{2}\times \nu\times Z([1,\nu^{3}])\cong Z([1,\nu^{3}])\times\nu^{2}\times \nu$. Thus it is the unique irreducible submodule of $Z([1,\nu^{3}])\times Q([\nu,\nu^{2}])$. Using Lemma \ref{duality lemma} and then taking the contragredients we get that $\theta$ is the unique irreducible quotient of $Q([\nu,\nu^{2}])\times Z([1,\nu^{3}])$. Now doing a similar calculation as in case (a) by taking $\mu_{1}=Q([\nu,\nu^{2}])$ and $\mu_{2}=Z([1,\nu^{3}])$ we get that $Q([\nu,\nu^{2}])\times Z([1,\nu^{3}])$, and hence $\theta$, doesn't have a symplectic period.

If $\theta=Z([1,\nu^{2}],[\nu,\nu^{3}])$,  by Theorem A. 10(iii) of \cite{Tadic1}, $\newline$$\theta=Q([1,\nu],[\nu,\nu^{2}],[\nu^{2},\nu^{3}])$. Twisting $\theta$ by an appropriate power of $\nu$ makes it a unitary representation. By Theorem \ref{unitary symplectic} it doesn't have a symplectic period. 

If $\theta=Z([1,\nu^{3}],[\nu,\nu^{2}])\cong Z([1,\nu^{3}])\times Z([\nu,\nu^{2}])$, it has a symplectic period by Lemma \ref{hereditary for Gl6}. 

Thus we are done with case (b).

\paragraph{{\bf Case c) $\pi=Z([1,\nu])\times Z([\nu^{2},\nu^{3}])\times Z([\nu^{4},\nu^{5}])$ }}

$\newline$In this case, all irreducible subquotients of $\pi$ are $Z([1,\nu],[\nu^{2},\nu^{3}],[\nu^{4},\nu^{5}])$, $Z([1,\nu^{3}],[\nu^{4},\nu^{5}])$, $Z([1,\nu],[\nu^{2},\nu^{5}])$ and $Z([1,\nu^{5}])$. Of these, by Lemma \ref{juxt no symplectic period lemma}, $Z([1,\nu^{3}],[\nu^{4},\nu^{5}])$ and $Z([1,\nu],[\nu^{2},\nu^{5}])$ do not have a symplectic period. $Z([1,\nu^{5}])$ being a character clearly has a symplectic period. We analyze the remaining representation.

If $\theta=Z([1,\nu],[\nu^{2},\nu^{3}],[\nu^{4},\nu^{5}])$, by definition, it is the unique irreducible submodule of $Z([\nu^{4},\nu^{5}])\times Z([\nu^{2},\nu^{3}])\times Z([1,\nu])$. Thus we get that it is the unique irreducible quotient of $\pi=Z([1,\nu])\times Z([\nu^{2},\nu^{3}])\times Z([\nu^{4},\nu^{5}])$ (using Lemma \ref{duality lemma} and then taking the contragredients). Since $Z([\nu^{2},\nu^{3}],[\nu^{4},\nu^{5}])$ is a quotient of $Z([\nu^{2},\nu^{3}])\times Z([\nu^{4},\nu^{5}])$, we get $\theta$ is a quotient of $Z([1,\nu])\times Z([\nu^{2},\nu^{3}],[\nu^{4},\nu^{5}])$. By taking $\mu_{1}=Z([1,\nu])$ and $\mu_{2}=Z([\nu^{2},\nu^{3}],[\nu^{4},\nu^{5}])$, and doing a similar calculation as in case (a), we get that $\newline$$Z([1,\nu])\times Z([\nu^{2},\nu^{3}],[\nu^{4},\nu^{5}])$ and hence $\theta$, doesn't have a symplectic period.

Thus we are done with case (c).

\paragraph{{\bf Case d) $\pi=Z([1,\nu])\times Z([\nu,\nu^{2}])\times Z([\nu,\nu^{2}])$ }}
$\newline$In this case, all irreducible subquotients of $\pi$ are $Z([1,\nu],[\nu,\nu^{2}],[\nu,\nu^{2}])$ and $Z([1,\nu^{2}],[\nu],[\nu,\nu^{2}])$. We now analyze both of these representations one by one.

If $\theta=Z([1,\nu^{2}],[\nu],[\nu,\nu^{2}])\cong Z([1,\nu^{2}])\times Z([\nu,\nu^{2}])\times \nu$, by Theorem 3.7 of \cite{Offen-Sayag2}, it has a mixed Klyachko model. Hence by Theorem \ref{strong uniqueness}, it is not ${\rm Sp}_{3}(F)$-distinguished.

If $\theta=Z([1,\nu],[\nu,\nu^{2}],[\nu,\nu^{2}])$, it is the unique irreducible submodule of $Z([\nu,\nu^{2}])\times Z([\nu,\nu^{2}])\times Z([1,\nu])$. Thus it is the unique irreducible submodule of $Z([\nu,\nu^{2}])\times Z([1,\nu],[\nu,\nu^{2}])\cong Z([\nu,\nu^{2}])\times Q([1,\nu],[\nu,\nu^{2}])$ (by Example 11.4 in \cite{Z}). By Theorem 1 of \cite{Minguez-Lapid}, this representation is irreducible and so $\theta\cong Z([\nu,\nu^{2}])\times Z([1,\nu],[\nu,\nu^{2}])$ and has a symplectic period by Lemma \ref{hereditary for Gl6}.

Thus we are done with case (d).

\paragraph{{\bf Case e) $\pi=Z([1,\nu])\times Z([\nu^{2},\nu^{3}])\times Z([\nu^{2},\nu^{3}])$ }}
$\newline$In this case, all irreducible subquotients of $\pi$ are $Z([1,\nu],[\nu^{2},\nu^{3}],[\nu^{2},\nu^{3}])$ and $Z([1,\nu^{3}],[\nu^{2},\nu^{3}])$. We now analyze both of these representations one by one.

If $\theta=Z([1,\nu^{3}],[\nu^{2},\nu^{3}])\cong Z([1,\nu^{3}])\times Z([\nu^{2},\nu^{3}])$, it has a symplectic period by Lemma \ref{hereditary for Gl6}.

If $\theta=Z([1,\nu],[\nu^{2},\nu^{3}],[\nu^{2},\nu^{3}])$, it is the unique irreducible submodule of $Z([\nu^{2},\nu^{3}])\times Z([\nu^{2},\nu^{3}])\times Z([1,\nu])$. Now, $Z([\nu^{2},\nu^{3}])\times Z([\nu^{2},\nu^{3}])\times Z([1,\nu])$ is glued from $Z([\nu^{2},\nu^{3}])\times Z([1,\nu],[\nu^{2},\nu^{3}])$ and $Z([\nu^{2},\nu^{3}])\times Z([1,\nu^{3}])$. Using Theorem 1 of \cite{Minguez-Lapid}, we get that $Z([\nu^{2},\nu^{3}])\times Z([1,\nu],[\nu^{2},\nu^{3}])\cong Z([\nu^{2},\nu^{3}])\times Q([1],[\nu,\nu^{2}],[\nu^{3}])$ is irreducible and so $\theta\cong Z([\nu^{2},\nu^{3}])\times Z([1,\nu],[\nu^{2},\nu^{3}])$. Thus $\tilde{\theta}\cong Z([\nu^{-3},\nu^{-2}])\times Z([\nu^{-1},1],[\nu^{-3},\nu^{-2}])$. Now doing a similar calculation as in case (a) by taking $\mu_{1}=Z([\nu^{-3},\nu^{-2}])$ and $\mu_{2}=Z([\nu^{-1},1],[\nu^{-3},\nu^{-2}])$, we get that $\tilde{\theta}=\mu_{1}\times\mu_{2}$ doesn't have a symplectic period. Thus even $\theta$ is not ${\rm Sp}_{3}(F)$-distinguished.

Thus we are done with case (e).

Notice that in cases (f),(g),(h) all the irreducible subquotients of $\pi$ are twists of the contragredients of the ones obtained in cases (a),(d),(e) respectively. Hence the only subquotients with a symplectic period are up to a twist, duals of the ones already obtained previously. 

\paragraph{{\bf Case 2: All the three $\pi_{i}$'s are twists of Steinberg.}}$\newline$
The representations that we are looking at in this case are of the form $$\pi=Q([\chi_{1},\chi_{1}\nu])\times Q([\chi_{1}^{'},\chi_{1}^{'}\nu])\times Q([\chi_{1}^{''},\chi_{1}^{''}\nu]).$$We first prove a lemma which is going to be repeatedly used in the analysis of representations in this case.
\begin{lemma}\label{no symplectic lemma quotient}
Let $\pi=Q([\nu^{a},\nu^{a+1}])\times Q([\nu^{b},\nu^{b+1}])\times Q([\nu^{c},\nu^{c+1}])$. Then $\pi$ has a symplectic period only if $a=b=c+1$.
\end{lemma}
\begin{proof}
Let $\mu_{1}=Q([\nu^{a},\nu^{a+1}])$ and $\mu_{2}=Q([\nu^{b},\nu^{b+1}])\times Q([\nu^{c},\nu^{c+1}])$. Since $\mu_{1}$ doesn't have a symplectic period (by Theorem \ref{strong uniqueness}), $${\rm Hom}_{H}\big({\rm ind}^{H}_{H_{2,0}}(\delta^{1/2}_{P_{2,4}}\mu_{1}\otimes\mu_{2}|_{H_{2,0}}),\mathbb C\big)={\rm Hom}_{{\rm Sp}_{1}(F)}\big(\mu_{1},1\big)\otimes {\rm Hom}_{{\rm Sp}_{2}(F)}\big(\mu_{2},\mathbb C\big)=0$$(by Eqn. \ref{general case eqn}). Thus the other term ${\rm Hom}_{H}\big({\rm ind}^{H}_{H_{2,2}}(\delta^{1/2}_{P_{2,4}}\mu_{1}\otimes\mu_{2}|_{H_{2,2}}),\mathbb C\big)$ has to be non-zero. Now, $r_{(2,2),(4)}(Q([\nu^{b},\nu^{b+1}])\times Q([\nu^{c},\nu^{c+1}])$ is glued from $Q([\nu^{b},\nu^{b+1}])\otimes Q([\nu^{c},\nu^{c+1}])$, $Q([\nu^{c},\nu^{c+1}])\otimes Q([\nu^{b},\nu^{b+1}])$ and $(\nu^{b+1}\times\nu^{c+1})\otimes(\nu^{b}\times\nu^{c})$ (by Lemma 2.12 of \cite{BZ2}). It can be checked easily that substituting $r_{(2,2),(4)}(\mu_{2})$ by the first two of the three representations makes this Hom space 0. Thus, $${\rm Hom}_{{\rm GL}_{2}(F)\times {\rm Sp}_{1}(F)}\big(\nu^{-1}Q([\nu^{a},\nu^{a+1}]) \otimes (\nu^{b+1}\times\nu^{c+1})\otimes(\nu^{b}\times\nu^{c}),\mathbb C\big)\neq 0.$$Solving the equations for this to be non-zero gives the lemma. 
\end{proof}
By similar arguments using Lemma \ref{irr} as in case 1 it can be easily concluded that if there is at most one link among the three segments then $\pi$ doesn't have an $H$-distinguished subquotient. Thus we look at the case where there are at least two links among the segments. Since twisting by a character doesn't matter to us, without loss of generality we can assume $\chi_{1}$ to be trivial. As earlier following are then the 8 possible cases:
$\newline a) \ Q([1,\nu])\times Q([\nu,\nu^{2}])\times Q([\nu^{3},\nu^{4}])$
$\newline b) \ Q([1,\nu])\times Q([\nu,\nu^{2}])\times Q([\nu^{2},\nu^{3}])$
$\newline c) \ Q([1,\nu])\times Q([\nu^{2},\nu^{3}])\times Q([\nu^{4},\nu^{5}])$
$\newline d) \ Q([1,\nu])\times Q([\nu,\nu^{2}])\times Q([\nu,\nu^{2}])$
$\newline e) \ Q([1,\nu])\times Q([\nu^{2},\nu^{3}])\times Q([\nu^{2},\nu^{3}])$
$\newline f) \ Q([1,\nu])\times Q([\nu^{2},\nu^{3}])\times Q([\nu^{3},\nu^{4}])$
$\newline g) \ Q([1,\nu])\times Q([1,\nu])\times Q([\nu,\nu^{2}])$
$\newline h) \ Q([1,\nu])\times Q([1,\nu])\times Q([\nu^{2},\nu^{3}])$
$\newline$
In each one of the cases listed above, we evaluate all possible irreducible subquotients (using Theorem 7.1 of \cite{Z}) and individually determine whether they have a symplectic period or not.
\paragraph{{\bf Case a) $\pi=Q([1,\nu])\times Q([\nu,\nu^{2}])\times Q([\nu^{3},\nu^{4}])$ }}
$\newline$In this case, all irreducible subquotients of $\pi$ are $Q([1,\nu],[\nu,\nu^{2}],[\nu^{3},\nu^{4}])$, $Q([1,\nu^{2}],[\nu],[\nu^{3},\nu^{4}])$, $Q([1,\nu^{4}],[\nu])$ and $Q([1,\nu],[\nu,\nu^{4}])$. We now analyze each of these representations one by one.

If $\theta=Q([1,\nu],[\nu,\nu^{2}],[\nu^{3},\nu^{4}])$, it is the unique irreducible quotient of $Q([\nu^{3},\nu^{4}])\times Q([\nu,\nu^{2}])\times Q([1,\nu])$ which doesn't have a symplectic period by Lemma \ref{no symplectic lemma quotient}. Hence $\theta$ doesn't have one.

If $\theta=Q([1,\nu^{2}],[\nu],[\nu^{3},\nu^{4}])$, it is the unique irreducible quotient of $Q([\nu^{3},\nu^{4}])\times \nu \times Q([1,\nu^{2}])$. Thus it is an irreducible quotient of $Q([\nu^{3},\nu^{4}])\times Q([1,\nu])\times Q([\nu,\nu^{2}])$ (by Theorem \ref {product of quotient}) which doesn't have a symplectic period by Lemma \ref{no symplectic lemma quotient}. Hence $\theta$ doesn't have one.

If $\theta=Q([1,\nu^{4}],[\nu])\cong Q([1,\nu^{4}])\times \nu$ (by Proposition 8.5 of \cite{Z}), it is generic and hence doesn't have a symplectic period (by Theorem \ref{strong uniqueness}).

If $\theta=Q([1,\nu],[\nu,\nu^{4}])$, it is a quotient of $Q([\nu,\nu^{2}])\times Q([\nu^{3},\nu^{4}])\times Q([1,\nu])$ which doesn't have a symplectic period by Lemma \ref{no symplectic lemma quotient}. Hence it doesn't have one too.

Thus we are done with case (a).

\paragraph{{\bf Case b) $\pi=Q([1,\nu])\times Q([\nu,\nu^{2}])\times Q([\nu^{2},\nu^{3}])$ }}
$\newline$In this case, all irreducible subquotients of $\pi$ are $Q([1,\nu],[\nu,\nu^{2}],[\nu^{2},\nu^{3}])$, $Q([1,\nu^{2}],[\nu],[\nu^{2},\nu^{3}])$, $Q([1,\nu],[\nu,\nu^{3}],[\nu^{2}])$, $Q([1,\nu^{3}],[\nu],[\nu^{2}])$, $Q([1,\nu^{2}],[\nu,\nu^{3}])$ and $Q([1,\nu^{3}],[\nu,\nu^{2}])$. We now analyze each of these representations one by one.

If $\theta=Q([1,\nu],[\nu,\nu^{2}],[\nu^{2},\nu^{3}])$, twisting $\theta$ by an appropriate power of $\nu$ makes it a unitary representation. By Theorem \ref{unitary symplectic} it doesn't have a symplectic period. 
 
If $\theta=Q([1,\nu^{2}],[\nu],[\nu^{2},\nu^{3}])$, it is the unique irreducible quotient of $Q([\nu^{2},\nu^{3}])\times \nu\times Q([1,\nu^{2}])$. Now this itself is a quotient of $Q([\nu^{2},\nu^{3}])\times Q([1,\nu])\times Q([\nu,\nu^{2}])$ which doesn't have a symplectic period by Lemma \ref{no symplectic lemma quotient}. Hence $\theta$ doesn't have one.

If $\theta=Q([1,\nu],[\nu,\nu^{3}],[\nu^{2}])$, it can be obtained by twisting the contragredient of $Q([1,\nu^{2}],[\nu],[\nu^{2},\nu^{3}])$ which as showed in the last paragraph, doesn't have a symplectic period. 

If $\theta=Q([1,\nu^{3}],[\nu],[\nu^{2}])$, it is the unique irreducible quotient of $\nu^{2}\times \nu\times Q([1,\nu^{3}])$ and hence of $Z([\nu,\nu^{2}])\times Q([1,\nu^{3}])$. Now doing a similar calculation as in case 1(a) by taking $\mu_{1}=Z([\nu,\nu^{2}])$ and $\mu_{2}=Q([1,\nu^{3}])$ we get that $Z([\nu,\nu^{2}])\times Q([1,\nu^{3}])$, and hence $\theta$, doesn't have a symplectic period.

If $\theta=Q([1,\nu^{2}],[\nu,\nu^{3}])$, it has a symplectic period by Theorem \ref{unitary symplectic}.

If $\theta=Q([1,\nu^{3}],[\nu,\nu^{2}])\cong Q([1,\nu^{3}])\times Q([\nu,\nu^{2}])$, (by Proposition 8.5 of \cite{Z}), it is generic and hence doesn't have a symplectic period (by Theorem \ref{strong uniqueness}).
 
Thus we are done with case (b).

\paragraph{{\bf Case c) $\pi=Q([1,\nu])\times Q([\nu^{2},\nu^{3}])\times Q([\nu^{4},\nu^{5}])$ }}
$\newline$In this case, all irreducible subquotients of $\pi$ are $Q([1,\nu],[\nu^{2},\nu^{3}],[\nu^{4},\nu^{5}])$, $Q([1,\nu^{3}],[\nu^{4},\nu^{5}])$, $Q([1,\nu],[\nu^{2},\nu^{5}])$ and $Q([1,\nu^{5}])$. We now analyze each of these representations one by one.

If $\theta=Q([1,\nu^{5}])$, it is generic and hence doesn't have a symplectic period.

If $\theta=Q([1,\nu],[\nu^{2},\nu^{3}],[\nu^{4},\nu^{5}])$, it is the unique irreducible quotient of $Q([\nu^{4},\nu^{5}])\times Q([\nu^{2},\nu^{3}])\times Q([1,\nu])$ which doesn't have a symplectic period by Lemma \ref{no symplectic lemma quotient}. Hence it doesn't have one. The other two cases are dealt similarly.

If $\theta=Q([1,\nu^{3}],[\nu^{4},\nu^{5}])$, it is the unique irreducible quotient of $Q([\nu^{4},\nu^{5}])\times Q([1,\nu^{3}])$. Now this itself is a quotient of $Q([\nu^{4},\nu^{5}])\times Q([1,\nu])\times Q([\nu^{2},\nu^{3}])$ which doesn't have a symplectic period by Lemma \ref{no symplectic lemma quotient}. Hence $\theta$ doesn't have one. 

If $\theta=Q([1,\nu],[\nu^{2},\nu^{5}])$, it can be obtained by twisting the contragredient of $Q([1,\nu^{3}],[\nu^{4},\nu^{5}])$ which as showed in the last paragraph, doesn't have a symplectic period. 

Thus we are done with case (c).

\paragraph{{\bf Case d) $\pi=Q([1,\nu])\times Q([\nu,\nu^{2}])\times Q([\nu,\nu^{2}])$ }}
$\newline$In this case, all irreducible subquotients of $\pi$ are $Q([1,\nu],[\nu,\nu^{2}],[\nu,\nu^{2}])$ and $Q([1,\nu^{2}],[\nu],[\nu,\nu^{2}])$. We now analyze each of these representations one by one.

If $\theta=Q([1,\nu^{2}],[\nu],[\nu,\nu^{2}])\cong Q([1,\nu^{2}])\times \nu \times Q([\nu,\nu^{2}])$, it is generic and hence doesn't have a symplectic period. 

If $\theta=Q([1,\nu],[\nu,\nu^{2}],[\nu,\nu^{2}])$, it is the unique irreducible quotient of $Q([\nu,\nu^{2}])\times Q([\nu,\nu^{2}])\times Q([1,\nu])$. Thus it is the unique irreducible quotient of $Q([\nu,\nu^{2}])\times Q([1,\nu],[\nu,\nu^{2}])\cong Q([\nu,\nu^{2}])\times Z([1,\nu],[\nu,\nu^{2}])$ (by Example 11.4 in \cite{Z}). By Theorem 1 of \cite{Minguez-Lapid}, this representation is irreducible and so $\theta\cong Q([\nu,\nu^{2}])\times Z([1,\nu],[\nu,\nu^{2}])$. So it is a quotient of $Q([\nu,\nu^{2}])\times Z([1,\nu])\times Z([\nu,\nu^{2}])$. Now doing a similar calculation as in case 1(a) by taking $\mu_{1}=Q([\nu,\nu^{2}])$ and $\mu_{2}=Z([1,\nu])\times Z([\nu,\nu^{2}])$ we get that $Q([\nu,\nu^{2}])\times Z([1,\nu])\times Z([\nu,\nu^{2}])$, and hence $\theta$, doesn't have a symplectic period.

Thus we are done with case (d).

\paragraph{{\bf Case e) $\pi=Q([1,\nu])\times Q([\nu^{2},\nu^{3}])\times Q([\nu^{2},\nu^{3}])$ }}
$\newline$In this case, all irreducible subquotients of $\pi$ are $Q([1,\nu],[\nu^{2},\nu^{3}],[\nu^{2},\nu^{3}])$ and $Q([1,\nu^{3}],[\nu^{2},\nu^{3}])$. We now analyze each of these representations one by one.

If $\theta=Q([1,\nu^{3}],[\nu^{2},\nu^{3}])\cong Q([1,\nu^{3}])\times Q([\nu^{2},\nu^{3}])$, (by Proposition 8.5 of \cite{Z}), it is generic and hence doesn't have a symplectic period (by Theorem \ref{strong uniqueness}).

If $\theta=Q([1,\nu],[\nu^{2},\nu^{3}],[\nu^{2},\nu^{3}])$, it is the unique irreducible quotient of $Q([\nu^{2},\nu^{3}])\times Q([\nu^{2},\nu^{3}])\times Q([1,\nu])$. This doesn't have a symplectic period by Lemma \ref{no symplectic lemma quotient} and so $\theta$ doesn't have one too.

Thus we are done with case (e).

Notice that as earlier, in cases (f),(g),(h) all the irreducible subquotients of $\pi$ are twists of the contragredients of the ones obtained in cases (a),(d),(e) respectively. Hence the only subquotients with a symplectic period are up to a twist, duals of the ones already obtained previously. 

Rest five cases of $\pi_{1}\times\pi_{2}\times\pi_{3}$ are dealt similarly proving Theorem \ref{main theorem}. We just mention that no new $H$-distinguished are obtained from the other cases.
\section{conjectures for the general case}\label{conjectures}
Theorem \ref{Gl4 result} and Theorem \ref {main theorem} prompt us to make certain conjectures for the general $2n$ case. In order to do so we need to set up notation. 

Define, $\mathfrak G^{'}$ to be the set of all representations of ${\rm GL}_{2n}(F)$ of the form $Z(\Delta_{1},...,\Delta_{r})$ which satisfy the following properties:
\begin{enumerate}
\item All the segments are in the same cuspidal line.
\item Each segment is of even length.
\item No two segments have the beginning element in common.
\item Condition (1) and condition (3) implies that there is a natural ordering of the segments (with respect to the beginning element). Arrange $\Delta_{1},...,\Delta_{r}$ accordingly. We require that the intersection of each segment with its neighbors is odd in length, in particular is non-empty.
\end{enumerate}

The set $\mathfrak G^{'}$ is contained in the set of ladder representations of ${\rm GL}_{m}(F)$ as defined in \cite{Minguez-Lapid}.

Further define, $\mathfrak G \subset \cup_{i\geq 1} \textrm{Irr}({\rm GL}_{2i}(F))$ to be the set of all irreducible products of elements in $\mathfrak G^{'}$ i.e.
$$\mathfrak G=\{\pi_{1}\times\cdots\times\pi_{t}|\ \pi_{1},...,\pi_{t}\in \mathfrak G^{'} \ \textrm{and the product is irreducible}\}.$$
Let us now state the conjecture in the general case using the above notation.
\begin{conjecture}\label{main conjecture}
Let $\theta$ be an irreducible representation of ${\rm GL}_{2n}(F)$ carrying a symplectic period. Then there exists $\pi_{1},...,\pi_{t}\in \mathfrak G^{'}$ such that $$\theta \cong \pi_{1}\times\cdots\times\pi_{t}.$$ In other words, $\theta \in \mathfrak G$.\end{conjecture}
The following proposition verifies the conjecture for the unitary representations. 
\begin{proposition}
Let $\theta$ be an irreducible unitary representation having a symplectic period. Then $\theta \in \mathfrak G$.
\end{proposition}
\begin{proof}
By Theorem A. 10(iii) of \cite{Tadic1}, $U(\delta,t)$ where $\delta=Q([\rho\nu^{\frac{1-d}{2}},\rho\nu^{\frac{d-1}{2}}])$ is equal to $Z(\Delta_{1},...,\Delta_{d})$ where $$\Delta_{1}=[(\rho\nu^{\frac{1-d}{2}})\nu^{\frac{1-t}{2}},(\rho\nu^{\frac{1-d}{2}})\nu^{\frac{t-1}{2}}],\  \Delta_{2}=[(\rho\nu^{\frac{3-d}{2}})\nu^{\frac{1-t}{2}},(\rho\nu^{\frac{3-d}{2}})\nu^{\frac{t-1}{2}}],\cdots,$$$$\Delta_{d}=[(\rho\nu^{\frac{d-1}{2}})\nu^{\frac{1-t}{2}},(\rho\nu^{\frac{d-1}{2}})\nu^{\frac{t-1}{2}}].\hspace{3 in}$$The intersection of each segment with both its neighbors, in case they are arranged in the order of precedence, is of length $t-1$. So if $t$ is even, $U(\delta,t)\in\mathfrak G^{'}$. The proposition then follows from Theorem \ref{unitary symplectic}.
\end{proof}
The fact, $U(\delta,2m)\in\mathfrak G^{'}$ leads to an obvious question generalizing Proposition \ref{main prop}, which we state as the next conjecture.
\begin{conjecture}[$\bold{Hereditary \ Property}$]\label{conjecture Hereditary property}
Let $\theta\in\mathfrak G^{'}$. Then $\theta$ has a symplectic period. Moreover, if $\theta_{1},...,\theta_{d}\in\mathfrak G^{'}$ then $\theta_{1}\times\cdots\times\theta_{d}$ has a symplectic period.
\end{conjecture}
Conjecture \ref{main conjecture} and Conjecture \ref{conjecture Hereditary property} together imply that $\mathfrak G$ is precisely the set of $H$-distinguished representations of the linear groups. Thus Theorem \ref{Gl4 result} and Theorem \ref{main theorem} prove the conjectures for ${\rm GL}_{4}(F)$ and ${\rm GL}_{6}(F)$. Note that the above conjectures together imply that the property of having a symplectic period is dependent only on the combinatorial structure of the segments involved and not on the building blocks, i.e. the cuspidal representations. More precisely we state,
\begin{conjecture} \label{initial reduction conjecture}
Let $\pi\in Irr({\rm GL}_{2n}(F))$ be of the form $Z(\Delta_{1},...,\Delta_{r})$ such that all the segments are in the same cuspidal line. Let $\rho\in\ Irr({\rm GL}_{m}(F))$ be an element of the line. Let $\Delta^{'}_{i}$ be the segment obtained from $\Delta_{i}$ by replacing $\rho$ with the trivial representation of $F^{\times}$ and $\pi^{'}$ be the representation $Z(\Delta_{1}^{'},...,\Delta_{r}^{'})$ of ${\rm GL}_{2n/m}(F)$. Then,\\
1) If $2n/m$ is even then, $\pi$ has a symplectic period iff $\pi^{'}$ has a symplectic period.\\
2) If $2n/m$ is odd then $\pi$ doesn't have a symplectic period.
\end{conjecture}

\bibliographystyle{plain}

\end{document}